%% file: L3.tex
\documentclass[12pt,reqno]{amsart}
\usepackage{amssymb,amsmath}
\usepackage[all]{xy}
\usepackage[headings]{fullpage}
\usepackage{times}

\include{formatting}
\include{macros}

\begin{document}
\title{Explicit points on the Legendre curve III}

\author{Douglas Ulmer}
\address{School of Mathematics \\ Georgia Institute of Technology
  \\ Atlanta, GA 30332}
\email{ulmer@math.gatech.edu}

\date{\today}


\begin{abstract}
  We continue our study of the Legendre elliptic curve
  $y^2=x(x+1)(x+t)$ over function fields $K_d=\Fp(\mu_d,t^{1/d})$.
  When $d=p^f+1$, we have previously exhibited explicit points
  generating a subgroup $V_d\subset E(K_d)$ of rank $d-2$ and of
  finite, $p$-power index.  We also proved the finiteness of
  $\sha(E/K_d)$ and a class number formula:
  $[E(K_d):V_d]^2=|\sha(E/K_d)|$.  In this paper, we compute
  $E(K_d)/V_d$ and $\sha(E/K_d)$ explicitly as modules over
  $\Zp[\gal(K_d/\Fp(t))]$.
\end{abstract}

\maketitle

\section{Introduction}\label{s:intro}
Let $p$ be an odd prime number, $\Fp$ the field of $p$ elements, and
$K=\Fp(t)$ the rational function field over $\Fp$.  Let $E$ be the
elliptic curve over $K$ defined by $y^2=x(x+1)(x+t)$.  In
\cite{Ulmer14a} we studied the arithmetic of $E$ over the extension
fields $K_d=\Fp(\mu_d,t^{1/d})$ for integers $d$ not divisible by $p$.
In particular, when $d=p^f+1$ we exhibited explicit points generating
a subgroup $V_d\subset E(K_d)$ of rank $d-2$ and finite $p$-power
index.  Moreover, we showed that the Tate-Shafarevich group
$\sha(E/K_d)$ is finite and its order satisfies
$|\sha(E/K_d)|=[E(K_d):V_d]^2$.  Some of these results were
generalized to other values of $d$ in \cite{ConceicaoHallUlmer14}.

Our goal in this paper is to study the quotient group $E(K_d)/V_d$ and
the Tate-Shafarevich group $\sha(E/K_d)$ as modules over the group
ring $\Zp[\gal(K_d/K)]$.  In fact, we will completely determine both
modules in terms of combinatorial data coming from the action of the
cyclic group $\<p\>\subset(\Z/d\Z)^\times$ on the set $\Z/d\Z$.
Stating the most precise results requires some preliminaries which are
given in the next section, so in this introduction we state only the
main qualitative results.

\begin{thm}\label{thm:main}
  Let $p$ be an odd prime number and let $d=p^f+1$.  Let $K=\Fp(t)$,
  $K_d=\Fp(\mu_d,u)$ where $u^d=t$, and $G=\gal(K_d/K)$.  Let $E$ be
  the elliptic curve over $K$ defined by $y^2=x(x+1)(x+t)$.  Let $V_d$
  be the subgroup of $E(K_d)$ generated by the point
  $P=(u,u(u+1)^{d/2})$ and its conjugates by $G$.  Let $\sha(E/K_d)$
  be the Tate-Shafarevich group of $E$ over $K_d$.  Then $E(K_d)/V_d$
  and $\sha(E/K_d)$ are finite abelian $p$-groups with the following
  properties:
\begin{enumerate}
\item $E(K_d)/V_d$ and $\sha(E/K_d)$ are trivial if and only if
  $f\le2$.
\item The exponent of the group $E(K_d)/V_d$ is
  $p^{\lfloor(f-1)/2\rfloor}$.  The exponent of the group $\sha(E/K_d)$
    is $p^{\lfloor f/3\rfloor}$.  Here $\lfloor x\rfloor$ is the
    greatest integer $\le x$.
\item $\left(E(K_d)/V_d\right)^2$ and $\sha(E/K_d)$ are isomorphic as
  $\Zp[G]$-modules if and only if $f\le 4$.  If $f>4$, they are not
  isomorphic as abelian groups.
\item The Jordan-H\"older factors of $\sha(E/K_d)$ as
    $\Zp[G]$-module are the same as those of $E(K_d)/V_d$ with
    multiplicities doubled.
\item There is a polynomial $F_f(T)\in\Z[1/2][T]$ depending on $f$ but
  independent of $p$ such that 
$$|\sha(E/K_d)|=p^{F_f(p)}$$
for all $p>2$.
\end{enumerate}
\end{thm}

Part (4) of the theorem may be viewed as an analogue of the Gras
conjecture, see \cite{Gras77} and \cite{MazurWiles84}. 

To my knowledge, the phenomenon of ``interpolation in $p$'' in part
(5) has not been observed before.  In fact, even more is true, namely
that all of the invariants of $\sha(E/K_d)$ and $E(K_d)/V_d$ as
abelian $p$-groups (i.e., the order of their $p^a$-torsion subgroups
for all $a$) are described by polynomials independent of $p$.

Results on the exact structure of $E(K_d)/V_d$ and $\sha(E/K_d)$ as
$\Zp[G]$-modules will be stated in Section~\ref{s:refined-results}
after some preliminaries in Section~\ref{s:orbits}.

In fact, we will prove results on the discriminant of the ``new part''
of $E(K_d)$ with its height pairing and on the $\Zp[G]$-module
structure of the ``new part'' of $\sha(E/K_d)$ for any $d$ such that
$p$ is balanced modulo $d$ in the sense of
\cite[Def.~2.1]{ConceicaoHallUlmer14}.  (This is the situation in
which there are points on $E(K_d)$ not coming from $E(K_e)$ for $e$ a
proper divisor of $d$.)  In cases where we have explicit points
(namely for $d=p^f+1$ as in \cite{Ulmer14a} or $d=2(p^f-1)$ as in
\cite{ConceicaoHallUlmer14}) we obtain good control on $E(K_d)/V_d$ as
well.  Some of our results apply to other curves and their Jacobians,
and for $p=2$.  See Theorems~\ref{thm:disc}, \ref{thm:index}, and
\ref{thm:sha} for the main refined results.

The two key ideas that afford such strong control on Mordell-Weil and
Tate-Shafarevich groups are: (i) that the N\'eron model of $E$ over
$\P^1_{/\Fp(\mu_d)}$ is dominated by a product of curves; and (ii)
ideas of Shioda and Dummigan which allow us to use crystalline
cohomology to compute Tate cycles and Brauer groups for products of
curves.  Similar ideas were used by Dummigan in \cite{Dummigan95} and
\cite{Dummigan99} to compute the discriminant of the Mordell-Weil
lattice and the structure of the Tate-Shafarevich group for a constant
supersingular elliptic curve over the function field of a Hermitian
curve.  In our case, the group of symmetries (essentially $G$ above)
is much smaller, the representation theory is much simpler, and as a
result we are able to boil the combinatorics down to very explicit
statements.

Here is an outline of the rest of the paper: In
Section~\ref{s:orbits}, we consider the orbits of $\<
p\>\subset(\Z/d\Z)^\times$ acting on $\Z/d\Z$.  These orbits index
certain $\Zp[G]$-modules which we use to decompose and describe
$E(K_d)$ and $\sha(E/K_d)$.  In Section~\ref{s:refined-results} we
state the more precise results on $E(K_d)$ and $\sha(E/K_d)$ alluded to
above.  In Section~\ref{s:DPC} we work out the geometry relating the
N\'eron model of $E$ to a product of curves (which in fact are Fermat
quotient curves) and the relations between the Mordell-Weil and
Tate-Shafarevich groups of $E$ and the N\'eron-Severi and Brauer
groups of the product of curves.  In
Section~\ref{s:arithmetic-of-CxD}, we work out the N\'eron-Severi
group and the $p$-part of the Brauer group of a general product of
curves in terms of crystalline cohomology.  That this is possible (in
the context of supersingular surfaces) was noted by Shioda in
\cite{Shioda91} and developed more fully by Dummigan in
\cite{Dummigan95}.  We use a somewhat different method than Dummigan
did, yielding more general results, although his results would suffice
for our application to the Legendre curve.  In Section~\ref{s:H1(C)}
we collect results on the cohomology of the curves appearing in the
product mentioned above. These results give the raw material for
Section~\ref{s:exercises}, where we carry out the $p$-adic exercises
needed to compute $E(K_d)$ and $\sha(E/K_d)$.  In
Section~\ref{s:proofs} we put all the pieces together and prove the
main results. Finally, Section~\ref{s:complements} contains various
generalizations and complements.

It is a pleasure to thank the anonymous referee for a very careful
reading of the paper and several valuable suggestions.

\section{Orbits, Invariants, Representations}\label{s:orbits}

Throughout this section, $p$ is an arbitrary prime number and $d$ is a
positive integer not divisible by $p$.  We write $(\Z/d\Z)^\times$ for
the multiplicative group modulo $d$ and $\<p\>$ for the cyclic
subgroup generated by $p$.

\subsection{Orbits}
Consider the action of $(\Z/d\Z)^\times$ on the set $\Z/d\Z$ by
multiplication.  By restriction, the subgroup $\<p\>$
acts on $\Z/d\Z$.  We write $\tilde O=\tilde O_{d,p}$ for the set of orbits.  Thus,
if $o\in\tilde O$ and $i\in o\subset\Z/d\Z$, then
$o=\{i,pi,p^2i,\dots\}$.  

Clearly the orbit through $0\in\Z/d\Z$ is a singleton $\{0\}$.  If $d$
is even (and therefore $p$ is odd), then the orbit through $d/2$ is
also a singleton because $p(d/2)=(d/2)$ in $\Z/d\Z$.  For reasons
which will become apparent later, we will usually exclude these two
orbits, and we define
$$O=O_{d,p}=\begin{cases}
\tilde O\setminus\left\{\{0\}\right\}&\text{if $d$ is odd}\\
\tilde O\setminus\left\{\{0\},\{d/2\}\right\}&\text{if $d$ is even.}
\end{cases}$$

Note that if $o\in\tilde O$, then $\gcd(i,d)$ is the same for all
$i\in o$, and we write $\gcd(o,d)$ for this common value.  It will
sometimes be convenient to consider only orbits with $\gcd(o,d)=1$
(which one might call ``new'' orbits), so we define
$$O'=O'_{d,p}=\left\{o\in O|\gcd(o,d)=1\right\}.$$
Note that $O'_{d,p}$ is just the set of cosets of $\<p\>$ in
$(\Z/d\Z)^\times$.  Note also that the set of orbits $o\in O$ with
$\gcd(o,d)=e$ for a fixed $e<d/2$ is in bijection with $O'_{d/e,p}$.

\subsection{Balanced orbits}\label{ss:balanced}
From here through the end of Subsection~\ref{ss:computing-invariants}
we assume that $d>2$ so that $O_{d,p}$ is not empty.

As in \cite{ConceicaoHallUlmer14}, we divide $(\Z/d\Z)^\times$ into two
subsets $A$ and $B$ where $A$ (resp.~$B$) consists of those classes
with least residue in the interval $(0,d/2)$ (resp.~in $(d/2,d)$).

We say that an orbit $o$ is \emph{balanced} if we have $|o\cap
A|=|o\cap B|$, and we say $d$ is \emph{balanced modulo $p$} if every
orbit $o\in O'_{d,p}$ is balanced. For example, by \cite[5.4,
5.5]{ConceicaoHallUlmer14}, $d$ is balanced modulo $p$ if $d$ divides
$p^f+1$ or if $d$ divides $2(p^f-1)$ and the ratio $2(p^f-1)/d$ is
odd.

\subsection{Invariants of orbits}\label{ss:invs-of-orbits}
Associated to each orbit $o$ we form a word on the two letter alphabet
$\{u,l\}$ ($u$ for upper and $l$ for lower) as follows: Choose a base
point $i$ so that the orbit $o=\{i,pi, p^2i,\dots,p^{|o|-1}i\}$.  The
associated word $w=w_1\cdots w_{|o|}$ is defined by
$$w_j=\begin{cases}
l&\text{if $-p^{j-1}i\in A$}\\
u&\text{if $-p^{j-1}i\in B$}.
\end{cases}$$
(The reason for the minus signs is explained in
Remark~\ref{rem:u-l-motivation}.)  Thus, for example, if $p=3$ and
$d=28$, the word associated to the orbit $\{6,18,26,22,10,2\}$ with
base point $6$ is $ullluu$

Note that $w$ depends on the choice of $i\in o$.  Changing the choice
of $i$ changes $w$ by a cyclic permutation of the letters.

Given a word $w=w_1\cdots w_{|o|}$, we define a sequence of integers
$a_j$ by $a_0=0$ and 
$$a_j=a_{j-1}+
\begin{cases}
1&\text{if $w_j=u$}\\
-1&\text{if $w_j=l$}.
\end{cases}$$
(So the word $w$ is viewed as a sequence of instructions to go up or
down.)

If $o$ is balanced, then the word $w$ associated to $o$ has as many
$u$'s as $l$'s and $a_{|o|}=0$.

\begin{defn}\label{def:base-pts}
We say the base point $i$ is \emph{good} if $a_j\ge0$
for $0\le j\le|o|$.  It is easy to see that every $o$ has a good base
point. The \emph{standard base point} for an orbit $o$ is the good base
point with smallest least positive residue.
\end{defn}

So for example, if $p=3$, $d=364$, and $o$ is the orbit $\{7, 21, 63,
189, 203, 245\}$, then there is a unique good base point, namely $7$,
with associated word $uuulll$.  On the other hand, if $o$ is the orbit
$\{37, 111, 333, 271, 85, 255\}$ then the good base points are $37$
(with word $uullul$) and $85$ (with word $uluull$), and the standard
base point is $37$.  From now on, given an orbit we choose the
standard base point and form the word associated to that base point.
This yields a well-defined function from orbits to words.  (It will be
essential below to choose a good base point, but which good base point
is chosen is of no import.  We introduce the notion of standard base
point simply for convenience.)

Now suppose that $w$ is the word associated to a balanced orbit $o$.
Then the first letter of $w$ must be $u$ and the last must be $l$, so
we can write $w$ in exponential form
$$w=u^{e_1}l^{e_2}\cdots l^{e_{2k}}$$ 
where each $e_j>0$.

\subsection{The complementary case}
Suppose that $d>2$ and $d$ divides $p^f+1$ for some $f$ so that
$-1\in\<p\>$.  If $i\in A$, then $p^fi\in B$ and conversely.  It
follows that if $o\in O_{d,p}$ and $w$ is the associated word, then
the second half of $w$ is the ``complement'' of the first half, i.e.,
each $u$ is replaced with an $l$ and each $l$ is replaced with a $u$.
More formally, if $w=w_1w_2\dots w_{|o|}$, then
$\{w_j,w_{|o|/2+j}\}=\{u,l\}$ for all $1\le j\le |o|/2$.

A similar discussion applies when $d$ divides $2(p^f-1)$ with an odd
quotient and $o$ is an orbit with $\gcd(o,d)$ odd.  Indeed, in this
case $p^f\equiv1+d/2\pmod d$ and $p^f$ is an element of order 2 in
$(\Z/d\Z)^\times$ which exchanges $A$ and $B$.  Thus if $o$ is an
orbit with $\gcd(o,d)$ odd, then the associated word has second half
equal to the complement of the first half.

These examples motivate the following definition.

\begin{defn} We say an orbit $o$ is \emph{complementary} if it is
  balanced and the associated word $w=w_1\dots w_{|o|}$ satisfies
  $\{w_j,w_{|o|/2+j}\}=\{u,l\}$ for $1\le j\le |o|/2$.
\end{defn}

If $o$ is complementary and we write the associated word in
exponential form $w=u^{e_1}l^{e_2}\cdots l^{e_{2k}}$, then
$e_{k+j}=e_j$.  Since the last letter must be $l$, the last letter of
the first half must be $u$ and so $k$ must be odd.

\subsection{Comparison with Dummigan's string diagrams}
In \cite{Dummigan95}, Dummigan introduces certain words on the
alphabet $\{X,O\}$ which he calls string diagrams.  He works entirely
in the context where $d=p^f+1$ (so all orbits are complementary), and
his diagrams are invariants of orbits closely related to our words
$w(o)$.  Indeed, given an orbit $o$ with base point $i$ and word
$w(o)$, the associated string diagram is $s=s_1\cdots s_f$ where
$$s_j=\begin{cases}O&\text{if $w_j=w_{j+1}$}\\
X&\text{if $w_j\neq w_{j+1}$.}
\end{cases}$$
He also defines circle diagrams by taking into account the rotations
induced by a change of base point.  It is easy to see that the map from words
to string diagrams is 2-to-1 and that we could phrase our arguments in
terms of Dummigan's string and circle diagrams.  However, for most of
our purposes, words as we have defined them are more convenient.

\subsection{More invariants}\label{ss:more-invs}
We continue to assume that $d>2$.  Let $o$ be a balanced orbit with
associated word $w$ written in exponential form as $w=u^{e_1}\cdots
l^{e_{2k}}$.  The exponents $e_1,\dots,e_{2k}$ give one invariant of
the orbit $o$.  

A second invariant of the orbit $o$ is its \emph{height}, defined as
$$ht(o)=\max\{e_1,e_1-e_2+e_3,\dots,e_1-e_2+e_3-\cdots+e_{2k-1}\}.$$ 
We may also describe the height as the maximum value of the function
$i\mapsto a_i$ defined above.  Note that in the complementary case, we
have $ht(o)=e_1-e_2+\cdots+e_k$.

We will define a third invariant in terms of invariant factors of
certain bi-diagonal matrices.  To that end, consider the integer,
$k\times k$, bi-diagonal matrix
$$B=B(e_1,\dots,e_{2k-1}):=\begin{pmatrix}
p^{e_1}&-p^{e_2}&0&\dots&\dots\\
0&p^{e_3}&-p^{e_4}&\dots&\dots\\
0&0&p^{e_5}&\dots&\dots\\
\vdots&\vdots&\vdots&\ddots\\
\vdots&\vdots&\vdots&&p^{e_{2k-1}}
\end{pmatrix}
$$
and define $d_1\le d_2\le\dots\le d_k$ as the exponents of the invariant factors
of $B$, so that $B$ can be transformed into
$$A=\begin{pmatrix}
p^{d_1}&0&0&\dots&\dots\\
0&p^{d_2}&0&\dots&\dots\\
0&0&p^{d_3}&\dots&\dots\\
\vdots&\vdots&\vdots&\ddots\\
\vdots&\vdots&\vdots&&p^{d_{k}}
\end{pmatrix}
$$
by a series of integer row and column operations.  We will discuss how
to compute these invariants in the next subsection.

\subsection{Computing invariant factors}\label{ss:computing-invariants}
We continue with the assumptions of the preceding subsection (so $o$
is a balanced orbit) and we give two algorithms for computing the
invariants $d_1,\dots,d_k$ attached to $o$.  This subsection is not
needed for the statements of the main results in
Section~\ref{s:refined-results}, so it may be skipped on a first
reading.

Roughly speaking, the first algorithm picks out $d_1$ and continues
inductively, while the second picks out $d_k$ and continues
inductively.  The second is more complicated than the first, but it
gives valuable information in the complementary case, see
Lemma~\ref{lemma:inv-factors3} and Remark~\ref{rem:equivalent-Bs}
below.  Both algorithms are based on the well-known fact that the
$i$-th invariant factor of a matrix $B$ is
$$\gcd(\text{$i\times i$ minors of $B$})/
\gcd(\text{$(i-1)\times (i-1)$ minors of $B$}).$$

To describe the results, we introduce the following notation: For
$1\le i\le j\le2k-1$, let $e_{ij}=e_i-e_{i+1}+e_{i+2}-\cdots\pm e_j$.
Also, we say that two matrices are \emph{equivalent} (notation:
$\sim$) if one can be transformed to the other by a series of integer
row and column operations.  

\begin{lemma}\label{lemma:inv-factors1}
  Assume $k>1$, let $e_1,\dots,e_{2k-1}$ be positive integers, and let
  $d_1,\dots,d_k$ be the integers attached as above to
  $B(e_1,\dots,e_{2k-1})$.  We have $d_1=\min\{e_1,\dots,e_{2k-1}\}$.
  Choose $i$ such that $d_1=e_i$ and define
$$B'=
\begin{cases}
  B(e_3,\dots,e_{2k-1})&\text{if $i=1$}\\
  B(e_1,\dots,e_{i-2},e_{i-1,i+1},e_{i+2},\dots,e_{2k-1})&\text{if $1<i<2k-1$}\\
  B(e_1,\dots,e_{2k-3})&\text{if $i=2k-1$.}
\end{cases}$$
Then $B(e_1,\dots,e_{2k-1})$ is equivalent to $(p^{d_1})\oplus B'$.
\end{lemma}

Note that we make no assumptions on the $e_i$ other than
positivity.  The result can thus be applied inductively to $B'$,
and thus gives an algorithm for computing all of the $d_j$.
For example, if $(e_1,\dots,e_{2k})=(4,1,3,5,4,3,5,4,2,1,2,6)$, then
the algorithm proceeds as follows:
\begin{align*}
B(4,1,3,5,4,3,5,4,2,1,2)&\underset{(i=2)}{\sim} (p^1)\oplus B(6,5,4,3,5,4,2,1,2)\\
B(6,5,4,3,5,4,2,1,2)&\underset{(i=7)}{\sim} (p^1)\oplus B(6,5,4,3,5,4,3)\\
B(6,5,4,3,5,4,3)&\underset{(i=4)}{\sim} (p^3)\oplus B(6,5,6,4,3)\\
B(6,5,6,4,3)&\underset{(i=5)}{\sim} (p^3)\oplus B(6,5,6)\\
B(6,5,6)&\underset{(i=2)}{\sim} (p^5)\oplus B(7)
\end{align*}
so the invariants $d_j$ are $1,1,3,3,5,7$.

\begin{proof}[Proof of Lemma~\ref{lemma:inv-factors1}]
  That $d_1=\min\{e_1,\dots,e_{2k-1}\}$ is evident from the
  description of $d_1$ as $\gcd\{p^{e_1},\dots,p^{e_{2k-1}}\}$.

  Write $B$ for $B(e_1,\dots,e_{2k-1})$.  If $i=1$, then $p^{e_1}$
  divides $-p^{e_2}$, and a single column operation transforms $B$
  into $(p^{e_1})\oplus B(e_3,\dots,e_{2k-1})$.  This is the desired
  result.

  Similarly, if $i=2k-1$, then $p^{e_{2k-1}}$ divides $-p^{e_{2k-2}}$,
  and a single row operation transforms $B$ into
  $B(e_1,\dots,e_{2k-3})\oplus (p^{e_{2k-1}})$.  This is the desired
  result.

  Now consider the case where $1<i<2k-1$ and assume that $i$ is odd.
  Then a row operation followed by a column operation transforms the
  submatrix
$$\begin{pmatrix}
-p^{e_{i-1}}&0\\p^{e_i}&-p^{e_{i+1}}
\end{pmatrix}$$
of $B$ into
$$\begin{pmatrix}
0&-p^{e_{i-1,i+1}}\\p^{e_i}&0
\end{pmatrix}$$
and leaves the rest of $B$ unchanged.  Permuting rows and columns
yields 
$$(p^{e_i})\oplus
B(e_1,\dots,e_{i-2},e_{i-1,i+1},e_{i+2},\dots,e_{2k-1}).$$

The case where $1<i<2k-1$ and $i$ is even is similar: We first
transform the submatrix
$$\begin{pmatrix}
p^{e_{i-1}}&-p^{e_i}\\0&p^{e_{i+1}}
\end{pmatrix}$$
of $B$ into
$$\begin{pmatrix}
0&-p^{e_i}\\p^{e_{i-1,i+1}}&0
\end{pmatrix}$$ 
and then permute rows and columns and multiply row 1 (containing
$-p^{e_i}$) by $-1$ to arrive at
$$(p^{e_i})\oplus
B(e_1,\dots,e_{i-2},e_{i-1,i+1},e_{i+2},\dots,e_{2k-1}).$$
This completes the proof of the lemma.
\end{proof}

\begin{lemma}\label{lemma:inv-factors2}\mbox{}
  Assume $k>1$, let $e_1,\dots,e_{2k-1}$ be positive integers, and let
  $d_1,\dots,d_k$ be the integers attached as above to
  $B(e_1,\dots,e_{2k-1})$.  We have
$$d_k=\max\{e_{ij}|1\le i\le j\le 2k-1,\text{$i$ and $j$ odd}\}.$$  
Choose $i\le j$ odd such that $d_k=e_{ij}$.  Define a subset
$T\subset\{1,2,3\}$ and matrices $B_\alpha$ for $\alpha\in S$ as
follows:
\begin{itemize}
\item $1\in T$ if and only if $i>1$.  If $i>1$, let
  $B_1=B(e_1,\dots,e_{i-2})$.
\item $2\in T$ if and only if $i<j$.  If $i<j$, let
  $B_2=B(e_{i+1},\dots,e_{j-1})^t$ ($t=$ transpose).
\item $3\in T$ if and only if $j<2k-1$.  If $j<2k-1$, let
  $B_3=B(e_{j+2},\dots,e_{2k-1})$.
\end{itemize}
Let $B'=\oplus_{\alpha\in T}B_\alpha$.  Then $B(e_1,\dots,e_{2k-1})$
is equivalent to $(p^{d_k})\oplus B'$.
\end{lemma}

Note that since we always choose a good base point for an orbit, if
$B(e_1,\dots,e_{2k-1})$ is the matrix attached to a balanced orbit
$o$, then the invariant $d_k$ is equal to the height of $o$.  We have
not emphasized this in the statement of the lemma, because the top
invariant factor of a general bidiagonal matrix (e.g., the matrices
$B_\alpha$ with $\alpha\in T$) need not be of the form $e_{1j}$.

This lemma applies equally well to lower-triangular bidiagonal
matrices, so it gives another inductive algorithm for computing all of
the $d_j$.  For example, if
$$(e_1,\dots,e_{2k-1})=(4,1,3,5,4,3,5,4,2,1,2),$$ 
then (ignoring transposes) the algorithm proceeds as follows:
\begin{align*}
B(4,1,3,5,4,3,5,4,2,1,2)&\underset{(i,j)=(1,7)}{\sim} 
(p^7)\oplus B(1,3,5,4,3)\oplus B(2,1,2)\\
B(1,3,5,4,3)&\underset{(i,j)=(3,3)}{\sim} (p^5)\oplus B(1)\oplus B(3)\\
B(2,1,2)&\underset{(i,j)=(1,3)}{\sim} (p^3)\oplus B(1)
\end{align*}
so the invariants $d_j$ are $1,1,3,3,5,7$.

\begin{proof}[Proof of Lemma~\ref{lemma:inv-factors2}]
  We write $B$ for $B(e_1,\dots,e_{2k-1})$.  The value of $d_k$ can be
  seen from the description of the invariant factors of $B$ in terms
  of minors. Indeed, note that 
$$\det B=p^{e_1+e_3+\cdots+e_{2k-1}}.$$  
On the other hand, the non-zero $(k-1)\times(k-1)$ minors of $B$ are
of two types: Those obtained by deleting row and column $i$ are of the
form $\pm\det B/p^{e_{2i-1}}$, and those obtained by deleting row $i$
and column $j$ with $j<i$ are of the form
$$\pm p^{e_1+e_3+\cdots+e_{2j-3}}p^{e_{2j}+e_{2j+2}+\cdots +e_{2i-2}}
p^{e_{2i+1}+\cdots+e_{2k-1}}.$$ 
It follows that $d_k$ is the maximum of $e_{ij}$ where $i\le j$ and
$i$ and $j$ are odd.  This is the first claim in the statement of the
Lemma.

To obtain the asserted equivalence, choose $i\le j$ odd such that
$d_k=e_{ij}$.  If $i>1$, then the definition of $e_{ij}$ implies
the inequalities:
\begin{align*}
e_{i-2,j}\le e_{ij}&\implies e_{i-2,i-1}\le0\\
e_{i-4,j}\le e_{ij}&\implies e_{i-4,i-1}\le0\\
&\quad\vdots\\
e_{1,j}\le e_{ij}&\implies e_{1,i-1}\le0.
\end{align*}
It follows that we may eliminate the entry $-p^{e_{i-1}}$ from $B$ by
a series of column operations.  More precisely, $B$ is equivalent to
$B(e_1,\dots,e_{i-2})\oplus B(e_i,\dots,e_{2k-1})$.  

Similarly, if $j<2k-1$, we have a series of inequalities $e_{ij}\ge
e_{ij+2}$, \dots, $e_{ij}\ge e_{i,2k-1}$ and these imply that by a
series of row operations we may eliminate $-p^{e_{j+1}}$, i.e., $B$ is
equivalent to $B(e_1,\dots,e_j)\oplus B(e_{j+2},\dots,e_{2k-1})$.

If $i>1$ and $j<2k-1$, then we may perform both of the procedures
above, so that
$$B\sim B(e_1,\dots,e_{i-2})\oplus B(e_i,\dots,e_j)\oplus
B(e_{j+2},\dots,e_{2k-1}).$$

If $i=j$, then $B(e_i)=(p^{d_k})$ and we are done.  

It remains to prove that if $i<j$, then $B(e_i,\dots,e_j)$ is
equivalent to $(p^{d_k})\oplus B(e_{i+1},\dots,e_{j-1})^t$.  To see
this, we note that the definition of $e_{ij}$ implies that
$e_{i\ell}\ge0$ and $e_{\ell j}\le0$ for all even $\ell$ with
$i<\ell<j$.  Using these inequalities, we transform
$B(e_i,\dots,e_j)$ by column operations into
$$\begin{pmatrix}
0&-p^{e_{i+1}}&0&\dots&\dots\\
0&p^{e_{i+2}}&-p^{e_{i+3}}&\dots&\dots\\
\vdots&\vdots&\vdots&\ddots\\
p^{d_k}&0&0&\cdots&p^{e_{j}}
\end{pmatrix}$$
then by transposing rows into
$$\begin{pmatrix}
p^{d_k}&0&0&\cdots&p^{e_{j}}\\
0&-p^{e_{i+1}}&0&\dots&\dots\\
0&p^{e_{i+2}}&-p^{e_{i+3}}&\dots&\dots\\
\vdots&\vdots&\vdots&\ddots\\
0&0&\cdots&p^{e_{j-2}}&-p^{e_{j-1}}\\
\end{pmatrix}$$ 
and finally by row operations and sign changes into $(p^{d_k})\oplus
B(e_{i+1},\dots,e_{j-1})^t$.

This completes the proof of the lemma.
\end{proof}

\begin{lemma}\label{lemma:inv-factors3}
  If $o$ is complementary \textup{(}so that $k$ is odd and
  $e_{k+i}=e_i$ for $1\le i\le k$\textup{)}, then we have
  $d_k=e_{1k}$, the other $d_j$ come in pairs \textup{(}i.e.,
  $d_1=d_2$, $d_3=d_4$, \dots\textup{)}, and
$$d_{k-1}=d_{k-2}=
\max\{e_{ij}|2\le i\le j\le k-1,\text{$i$ and $j$ even}\}.$$
\end{lemma}

\begin{proof}[Proof of Lemma~\ref{lemma:inv-factors3}]
  It is easy to see that $i=1$, $j=k$ achieves the maximum $e_{ij}$,
  so we have $d_k=e_{1k}=ht(o)$.  One application of
  Lemma~\ref{lemma:inv-factors2} shows that $B(e_1,\dots,e_{2k-1})$
  is equivalent to
$$p^{d_k}\oplus B(e_2,\dots,e_{k-1})^t\oplus B(e_2,\dots,e_{k-1}).$$
Thus the invariant factors $d_1,\dots,d_{k-1}$ come in pairs.
Applying the recipe of Lemma~\ref{lemma:inv-factors2} for the top
invariant factor to $B(e_2,\dots,e_{k-1})$ gives the assertion on
$d_{k-1}$ and $d_{k-2}$.
\end{proof}

\begin{rem}\label{rem:equivalent-Bs}
  Suppose that $e_1,\dots,e_{2k}$ are the exponents of a word coming
  from a good base point (so $e_{1,j}\ge0$ for all $j$) and suppose
  that $e_{1,2j+1}$ is maximum among $e_{1,\ell}$.  Then the following
  four matrices and their transposes all have the same invariant
  factors: $B(e_1,\dots,e_{2k-1})$, $B(e_2,\dots,e_{2k})$,
  $B(e_{2j+2},\dots,e_{2k},e_1,\dots,e_{2j})$ and
  $B(e_{2j+3},\dots,e_{2k},e_1,\dots,e_{2j+1})$.  Indeed (ignoring
  transposes), the first step of the second algorithm above shows that
  each of these matrices is equivalent to
$$(p^{e_{1,2j+1}})\oplus B(e_2,\dots,e_{2j})\oplus B(e_{2j+3},\dots,e_{2k-1}).$$
\end{rem}

\subsection{Representations of $G$}\label{ss:reps}
Fix an algebraic closure $\Fpbar$ of $\Fp$, and view $\mu_d$
as a subgroup of $\Fpbartimes$.
Let $W(\Fpbar)$ be the Witt vectors with coefficients in $\Fpbar$ and
let $\chi:\mu_d\to W(\Fpbar)$ be the Teichm\"uller character, so that
$\chi(\zeta)\equiv\zeta\pmod p$ for all $\zeta\in \mu_d$.  
Identifying $W(\Fpbar)$ with a subring of $\Qpbar$, the
$\Qpbar$-valued character group $\hat \mu_d$ of $\mu_d$ can be
identified with $\Z/d\Z$ by associating $\chi^i$ with $i$.

The group $\<p\>\subset(\Z/d\Z)^\times$ acts on $\mu_d$ via
exponentiation.  This yields an action on $\hat \mu_d\cong\Z/d\Z$ under
which $p$ acts by multiplication by $p$.  It is thus natural to
consider the set $\tilde O$ of orbits of $\<p\>$ on $\Z/d\Z$.
If $i\in\Z/d\Z$ and $o$ is the orbit of $\<p\>$ through $i$, then the
values of $\chi^i$ lie in the Witt vectors $W(\F_{p^{|o|}})$ and the
values of  $\sum_{i\in o}\chi^i$ lie in $\Zp=W(\Fp)$.

Now fix a finite extension $\Fq$ of $\Fp(\mu_d)$ in $\Fpbar$ and
let $G_1=\gal(\Fq/\Fp)$.  The action of $G_1$ on $\mu_d$ factors
through the homomorphism $G_1\to\<p\>$ which sends $\Fr_p$, the
$p$-power Frobenius, to $p$.

Let $G$ be the semi-direct product $\mu_d\sdp G_1$.  There is a
canonical identification 
$$G\cong\gal(\Fq K_d/K)=\gal(\Fq(u)/\Fp(t)).$$ 

To avoid confusion between number rings and group rings, we write $H$
for $\mu_d$.  Let $\Zp[H]$ and $\Zp[G]$ be the group rings of $H$ and
$G$ with coefficients in $\Zp$.  We also write $\Gamma=\Zp[H]$ which
we view as a $\Zp[H]$-module in the obvious way.  Letting $G_1$ act
on $\Gamma$ through its action on $H$ makes $\Gamma$ into a
$\Zp[G]$-module.

\begin{prop}\label{prop:G-reps}
\mbox{}\begin{enumerate}
\item There is a canonical isomorphism of $\Zp[H]$-modules
$$\Gamma=\bigoplus_{o\in \tilde O}\Gamma_o$$
where $\Gamma_o$ is a free $\Zp$-module of rank $|o|$ on which $H$
acts with character $\sum_{i\in o}\chi^i$.
\item For every orbit $o$, $\Gamma_o\subset\Gamma$ is stable under
  $\Zp[G]$ and $\Gamma_o\tensor\Qp$ is an absolutely irreducible
  $\Qp[G]$-module.
\item $\Gamma_o\tensor_{\Zp}\Fp$ is an absolutely irreducible $\Fp[G]$
  module.
\item If $o\neq o'$, then
  $\Gamma_o\tensor_{\Zp}\Qpbar\not\cong\Gamma_{o'}\tensor_{\Zp}\Qpbar$
  and 
  $\Gamma_o\tensor_{\Zp}\Fpbar\not\cong\Gamma_{o'}\tensor_{\Zp}\Fpbar$
  as $G$-modules.
\item Suppose that $\Fq$ is a finite extension of $\F_{p^{|o|}}$.
  Fix $i\in \Z/d\Z$, and let $o$ be the orbit of $\<p\>$
  through $i$.  Make the Witt vectors $W(\Fq)$ into a
  $\Zp[G]$-module by letting $\zeta\in\mu_d=H$ act by multiplication
  by $\zeta^i$ and letting $\Fr_p\in G_1\subset G$ act by the Witt-vector
  Frobenius.  Then we have an isomorphism of $\Zp[G]$-modules
$$W(\Fq)\cong\Gamma_o\tensor_{\Zp}\Zp[\gal(\Fq/\F_{p^{|o|}})].$$ 
\end{enumerate}
\end{prop}

\begin{proof}
For (1), since  $\sum_{i\in o}\chi^i$ takes values in $\Zp$, setting
$$\pi_o=(1/d)\sum_{h\in H}\left(\sum_{i\in
    o}\chi^{-i}(h)\right)h,$$
we have $\pi_o\in\Zp[H]$.
Orthogonality of characters implies that the elements $\pi_o$ form a
system of orthogonal idempotents: We have $1=\sum_{o\in \tilde
  O}\pi_o$ and $\pi_o\pi_{o'}=0$ if $o\neq o'$.  We define
$\Gamma_o=\pi_o\Gamma$.  This gives a direct sum decomposition
$\Gamma=\oplus_{o\in \tilde O}\Gamma_o$.  It follows from the
definition that $\Gamma_o$ is a free $\Zp$-module.  We may compute
its rank by noting that $\Gamma\tensor_\Zp\Qpbar$ decomposes under
$H$ into lines where $H$ acts by the characters $\chi^i$ with
$i\in\Z/d\Z$, and the subspace $\Gamma_o\tensor_\Zp\Qpbar$ is the
direct sum of the lines where $H$ acts by $\chi^i$ with $i\in o$, so
$\Gamma_o$ has $\Zp$-rank $|o|$.

For (2), since $g\pi_o=\pi_og$ for all $g\in\<p\>$, it
follows that $\Gamma_o$ is stable under $G$.  As an $H$-module,
$\Gamma_o\tensor_\Zp\Qpbar$ decomposes into lines where $H$ acts via
$\chi^i$ with $i\in o$, and $\<p\>$ permutes these lines
transitively, so $\Gamma_o$ is absolutely irreducible as $G$-module.

Part (3) follows from a similar argument, using that $d$ is relatively
prime to $p$, so the $\chi^i$ are distinct modulo $p$.

Part (4) follows immediately from a consideration of characters.

For (5), first consider the case where $\Fq=\F_{p^{|o|}}$.  Now
$W(\F_{p^{|o|}})$ is a cyclic $\Zp[G]$-module generated by $1$ and
with annihilator the left ideal generated by $[p^{|o|}]-1$ and $\prod_{i\in
  o}([h]-\chi^i(h))$ where $h$ is a generator of $H$.  Using this it
is easy to check that $1\mapsto\pi_o$ defines an isomorphism of
$\Zp[G]$-modules $W(\F_{p^{|o|}})\to\Gamma_o$.  The general case
follows from this and the normal basis theorem for $\Fq$ over
$\F_{p^{|o|}}$ (which yields an integral normal basis statement for
the corresponding extension of Witt rings).
\end{proof}

\begin{rem}
  If $M$ is a $\Zp[G]$-module, we write $M^o$ for $\pi_oM$.  By
  definition, $H$ acts on $M^o$ by characters $\chi^i$ with $i\in o$.
  Note, however, that it is \emph{not} clear \emph{a priori} what the
  action of $G_1$ is on $M^o$.  Indeed, the action of $G_1$ does not
  enter into the definition of $\pi_o$ and so we will have to determine
  the full action of $G$ on $M$ by other means.  The reason for not
  using $G_1$ in the definition of $\pi_o$ is that $p$ may divide the
  order of $G_1$, and we prefer to avoid the resulting complications in
  the representation theory of $G$.
\end{rem}

\begin{rem}\label{rem:E(K_d)-as-G-module}
We showed in \cite[Cor.~4.3]{Ulmer14a} that the group $V_d$ appearing
in Theorem~\ref{thm:main} is a cyclic module over $\Z[G]$ with
relations $2\sum_iP_i=2\sum_i(-1)^iP_i=0$.  It follows easily that
$V_d\tensor\Zp$ is isomorphic to 
$$\bigoplus_{o\in O_{d,p}}\Gamma_o.$$
Since $E(K_d)$ is a $G$-invariant superlattice of $V_d$, the absolute
irreducibility of $\Gamma_o$ noted above implies that we also have an
isomorphism of $\Zp[G]$-modules
$$E(K_d)\tensor\Zp\cong \bigoplus_{o\in O_{d,p}}\Gamma_o.$$
\end{rem}

\section{Refined results}\label{s:refined-results} 
In this section we state results on Mordell-Weil and Tate-Shafarevich
groups decomposed for the action of Galois.  These
imply the results stated in Theorem~\ref{thm:main}, and they also give
information in many other contexts.  The proofs will be given in
Section~\ref{s:proofs}.

Throughout, we fix a positive integer $d$ prime to $p$ and a finite
extension $\Fq$ of $\F_{p^{|o|}}$, and we set
$G=\gal(\Fq(u)/\Fp(t))$.  For the results on discriminants and
indices, the choice of $\Fq$ is not material, so we work over
$K_d=\Fp(\mu_d,u)$.  On the other hand, our results on
the Tate-Shafarevich group depend significantly on the choice of
$\Fq$. 

\subsection{Discriminants}
We have seen in \cite{ConceicaoHallUlmer14} that the ``new''  part of
$E(K_d)$ (i.e., the part not coming from $E(K_e)$ with
$e$ a proper divisor of $d$) is trivial if $p$ is not balanced modulo
$d$ and has rank $\phi(d)$ if $p$ is balanced modulo $d$.  In this
subsection we refine this result by breaking up $E(K_d)$ for the
action of $G$ and by computing the $p$-part of the discriminant of the
height pairing. 

Recall that $E(K_d)$ carries a canonical real-valued height pairing
which is non-degenerate modulo torsion.  (See, e.g.,
\cite[4.3]{Ulmer14b}.)  There is a rational-valued pairing
$\<\cdot,\cdot\>$ such that the canonical height pairing is
$\<\cdot,\cdot\>\log(|\Fp(\mu_d)|)$.  For convenience, we work with
the rational-valued pairing.  The group $E(K_d)\tensor\Zp$ inherits a
$\Qp$-valued pairing and the direct sum decomposition
$$E(K_d)\tensor\Zp\cong\bigoplus_{o\in
  O}\left(E(K_d)\tensor\Zp)\right)^o$$
is an orthogonal decomposition for this pairing.  We write
$\disc\left(E(K_d)\tensor\Zp\right)^o$ for the discriminant restricted
to one of the factors.  This is well-defined up to the square of a
unit in $\Zp$, but we will compute it only up to units.

Recall the sequence $a_0,\dots,a_{|o|}$ associated to $o$ in
Subsection~\ref{ss:invs-of-orbits} and the representation $\Gamma_o$
defined in Subsection~\ref{ss:reps}.

\begin{thmss}\label{thm:disc}\mbox{}\begin{enumerate}
\item We have an isomorphism of $\Zp[G]$-modules
$$\left(E(K_d)\tensor\Zp\right)^o\cong\begin{cases}
\Gamma_o&\text{If $\gcd(o,d)<d/2$ and $p$ is balanced modulo
  $d/\gcd(o,d)$}\\
0&\text{otherwise}
\end{cases}$$
\item If $\gcd(o,d)<d/2$ and $p$ is balanced modulo $d/\gcd(o,d)$,
  then up to a unit in $\Zp$ we have
$$\disc\left(E(K_d)\tensor\Zp\right)^o=
p^a$$
where $a={2\sum_{j=1}^{|o|}a_j}$.
\end{enumerate}
\end{thmss}

\subsection{Indices} 
Now we suppose that:

\begin{itemize}
\setlength{\itemindent}{1in}
\item[(a)] $d=p^f+1$ and $o\in O_{d,p}$ is any orbit;
\item[or (b)] $d=2(p^f-1)$ and $o\in O_{d,p}$ is such that
  $\gcd(o,d)$ is odd.   
\end{itemize}
In these cases, the orbit $o$ is complementary, and the word $w$
associated to each $o$ and may be written in exponential form
$$w=u^{e_1}l^{e_2}\cdots u^{e_k}l^{e_1}u^{e_2}\cdots l^{e_k}$$
where each $e_j>0$ and $k$ is odd.  In this case,
$ht(o)=e_1-e_2+\cdots+e_k$. 

Let $V_d\subset E(K_d)$ be the subgroup generated by the explicit points
as in \cite[8.3]{Ulmer14a} ($d=p^f+1$) or \cite[6.1]{ConceicaoHallUlmer14}
($d=2(p^f-1)$).

\begin{thmss}\label{thm:index} 
  Under the hypotheses \textup{(}a\textup{)} or \textup{(}b\textup{)}
  above we have an isomorphism of $\Zp[G]$-modules
$$\left(E(K_d)/V_d\right)^o\cong\Gamma_o/p^e$$
where $e=(f-ht(o))/2$.  When $\gcd(o,d)=1$,
$e=\sum_{j=1}^{(k-1)/2}e_{2j}$.
\end{thmss}

Under the assumptions of the theorem, it follows that $(E(K_d)/V_d)^o=0$
if and only if the word corresponding to $o$ has height $f$, and that
occurs only for words equivalent up to rotation to $u^fl^f$.

\subsection{Tate-Shafarevich groups}
Recall the integers $d_1,\dots,d_k$ attached to an orbit $o$ in
Subsection~\ref{ss:more-invs}. 

\begin{thmss}\label{thm:sha}
For any $d>2$ prime to $p$ and any $o\in O_{d,p}$, if
$\gcd(o,d)<d/2$ and $p$ is balanced modulo
  $d/\gcd(o,d)$ then:
\begin{enumerate}
\item There is an isomorphism of $\Zp[G]$-modules
$$\sha(E/\Fq(u))^o\cong\frac{\prod_{j=1}^k W_{d_j}(\Fq)}{W_{d_k}(\F_{p^{|o|}})}.$$
\item In particular, if $\Fq=\Fp(\mu_d)$ so that $\Fq(u)=K_d$,
  and $\gcd(o,d)=1$, then
$$\sha(E/K_d)^o\cong\prod_{j=1}^{k-1} W_{d_j}(\F_{p^{|o|}})
\cong\prod_{j=1}^{k-1}\Gamma_o/p^{d_j}.$$
\end{enumerate}
\end{thmss}

Under the assumptions of the theorem, it follows that
$\sha(E/\Fq(u))^o$ is trivial only when $\Fq=\Fp(\mu_d)$ and $k=1$,
and $k=1$ occurs if and only if the word associated to $o$ is
$u^fl^f$.

\section{Domination by a product of curves}\label{s:DPC}
In this section we relate the arithmetic of $E/\Fq(u)$ to that of a
suitable product of curves over $\Fq$.

\subsection{Basic data}
Fix an integer $d$ relatively prime to $p$, let $\Fq$ be a
finite extension of $\Fp(\mu_d)$, and let $G_1=\gal(\Fq/\Fp)$.

Let $\CC$ be the smooth, projective curve over $\Fp$ with affine model
$z^d=x^2-1$.  We write $P_{\pm}$ for the rational points $x=\pm1$,
$z=0$ on $\CC$.  Extending scalars, the group $\mu_2\times\mu_d$ acts
on $\CC\times_{\Fp}\Fq$ by multiplying the $x$ and $z$ coordinates by
roots of unity.  There is also an action of $G_1$ on
$\CC\times_\Fp\Fq$ via the factor $\Fq$.  Altogether we get an action
of $(\mu_2\times\mu_d)\sdp G_1$ on $\CC\times_\Fp\Fq$.  To simplify
notation, for the rest of this section we let $\CC$ denote the curve
over $\Fq$.

Let $\DD$ be the curve associated to $w^d=y^2-1$, so that $\DD$ is
isomorphic to $\CC$.  It has rational points $Q_{\pm}$ and an action
of $(\mu_2\times\mu_d)\sdp G_1$ defined analogously to those of $\CC$.

Let $\SS=\CC\times_{\Fq}\DD$ be the product surface.  We let the group
$\Delta:=\mu_2\times\mu_d$ act on $\SS$ ``anti-diagonally,'' i.e.,
with
$$(\zeta_2,\zeta_d)(x,y,z,w)=
\left(\zeta_2x,\zeta_2^{-1}y,\zeta_dz,\zeta_d^{-1}w\right).$$

Write $\NS(\SS)$ for the N\'eron-Severi group of $\SS$ and $\NS'(\SS)$
for the orthogonal complement in $\NS(\SS)$ of the subgroup generated
by the classes of the divisors $\CC\times\{Q_+\}$ and
$\{P_+\}\times\DD$.  (We could also describe $\NS'(\SS)$ as
$\divcorr\left((\CC,P_+),(\DD,Q_+)\right)$, the group of divisorial
correspondences between the two pointed curves, cf.~\cite[0.5.1 and
II.8.4]{Ulmer11}.)  The intersection form on $\NS(\SS)$ restricts to a
non-degenerate form on $\NS'(\SS)$.  The action of $\Delta$ on $\SS$
induces an action on $\NS'(\SS)$.  

Let $G=\mu_d\sdp G_1$.  We let $G$ act on $\SS$ via its action on
$\CC$; this yields an action of $G$ on $\NS'(\SS)$.  We let $G$ act on
$E(\Fq(u))$ via the identification $G\cong\gal(\Fq(u)/\Fp(t))$. 

The main result of this section relates the arithmetic of the Legendre
curve $E/\Fq(u)$ to that of $\SS$.

\begin{thm}\label{thm:reduction-to-S}
With notation as above:
  \begin{enumerate}
  \item There is a canonical isomorphism
$$E(\Fq(u))\tensor\Z[1/2d]\longisoto\left(\NS'(\SS)\tensor\Z[1/2d]\right)^\Delta$$
where the superscript $\Delta$ denotes the subgroup of invariants.
This isomorphism is compatible with the $G$-actions, and under it the
height pairing on $E(K)$ corresponds to the intersection pairing on
$\NS'(\SS)$.
\item There is a canonical isomorphism
$$\sha(E/\Fq(u))[p^\infty]\longisoto\Br(\SS)[p^\infty]^\Delta.$$
Here $\Br(\SS)$ is the \textup{(}cohomological\textup{)} Brauer group of $\SS$ and
$[p^\infty]$ means the $p$-torsion subgroup.  This
isomorphism is compatible with the $G$-actions.
  \end{enumerate}
\end{thm}

The rest of this section is devoted to a proof of the theorem and the
discussion of a mild generalization.  Note that the theorem for odd
values of $d$ follows from the case of even $d$ (by taking invariants
by a suitable subgroup of $G$), so \emph{for the rest of this section,
  we assume that $d$ is even}.

\subsection{The basic geometric result}
The main step in the proof of Theorem~\ref{thm:reduction-to-S} is to
relate the N\'eron model of $E/\Fq(u)$ to a suitable quotient of
$\SS$.  To that end, recall the Weierstrass fibration $\WW\to\P^1_u$
(whose fibers are the plane cubic reductions of $E$ at places of
$\Fq(u)$) and the N\'eron model $\EE\to\P^1_u$ which is obtained from
$\WW$ by blowing up singular points in the fibers over $u=0$,
$u\in\mu_d$, and $u=\infty$.  All this is discussed in detail in
\cite[\S7]{Ulmer14a}.

Note that since we are assuming that $d$ is even, $\CC$ has two points
at infinity which we denote $P'_{\pm}$ where the sign corresponds to
the limiting value of $x/z^{d/2}$.  Similarly, $\DD$ has two points at
infinity, denoted $Q'_{\pm}$.

Let $\tilde\SS=\widetilde{\CC\times\DD}$ be the blow up of $\SS$ at
the eight points $(P_{\pm},Q'_{\pm})$ and $(P'_{\pm},Q_{\pm})$.  These
points have stabilizers of order $d/2$ under the action of $\Delta$,
and they fall into two orbits, namely $\{(P_{\pm},Q'_{\pm})\}$ and
$\{(P'_{\pm},Q_{\pm})\}$, under the $\Delta$ action.  The action of
the stabilizer on the projectivized tangent space at each of these
points is trivial, so the action of $\Delta$ lifts canonically to
$\tilde\SS$ and the exceptional fibers are fixed pointwise by the
stabilizer of the corresponding point.  The action of $\Delta$ on
$\tilde\SS$ has other isolated fixed points, but we do not need to
make them explicit.

We let $\tilde\SS/\Delta$ denote the quotient of $\tilde\SS$ by the
action of $\Delta$.  This is a normal, projective surface with
isolated cyclic quotient singularities.  (They are in fact rational
double points, but we will not need this fact.)

Now we define a rational map $\SS\ratto\WW$ by requiring that
$$(x,y,z,w)\mapsto\left([X,Y,Z],u\right)=\left([z^d,xyz^d,1],zw\right)$$ 
where $([X,Y,Z],u)$ are the coordinates on a dense open subset of
$\WW$ as in \cite[\S7]{Ulmer14a}.  This induces a rational map
$\phi:\tilde\SS\ratto\WW$ which is obviously equivariant for the
$\Delta$ action, where $\Delta$ acts trivially on $\WW$.  Thus $\phi$
descends to a map on the quotient which we denote
$\psi:\tilde\SS/\Delta\ratto\WW$.

The following diagram shows the surfaces under consideration and
various morphisms between them:
$$\xymatrix{\llap{$\CC\times\DD=\ $}\SS&
  \tilde\SS\ar[l]_\rho\ar[d]_\pi\ar[rd]^\phi&&\cr
  &\tilde\SS/\Delta\ar[r]_\psi&\WW&\EE\ar[l]^\sigma\cr}$$ 
The quotient map $\pi$ is finite, and we will see just below that the
horizontal maps are birational morphisms.

\begin{prop}\label{prop:geometry}\mbox{}
  \begin{enumerate}
  \item The rational map $\phi$ is in fact a morphism.  Therefore
    $\psi$ is also a morphism and a birational isomorphism.
  \item $\phi$ contracts the strict transforms of  $P_{\pm}\times\DD$
    and $\CC\times Q'_{\pm}$ and is finite elsewhere.
  \item For generic $P\in\CC$, $\phi$ sends $P\times\DD$ to a
    bisection of $\WW\to\P^1$, where the two points in each fiber are
    inverse to one another.  Similarly, for generic $Q\in\DD$, $\phi$
    sends $\CC\times Q$ to a bisection of $\WW\to\P^1$, where the two
    points in each fiber are inverse to one another.
  \item The exceptional divisors over $P_{\pm}\times Q'_{\pm}$ map via
    $\phi$ to the torsion section $[0,0,1]$ of $\WW$, and the
    exceptional divisors over $P'_{\pm}\times Q_{\pm}$ map via $\phi$
    to the zero section $[0,1,0]$ of $\WW$.
  \end{enumerate}
\end{prop}

In part (3), ``$P$ generic'' means $P$ with trivial stabilizer, or
more explicitly $P\neq P_\pm,P'_\pm$ and $x(P)\neq0$.  Similarly for
``$Q$ generic.''

\begin{proof}
  It is easy to see that $\phi$ has generic degree $2d$ and it factors
  through quotient $\tilde\SS\to\tilde\SS/\Delta$ which is finite of
  degree $2d$.  This proves that $\psi$ is birational.  

  That $\phi$ is everywhere defined and has the stated geometric
  properties is a straightforward but tedious exercise in coordinates
  which we omit.  Since $\phi$ is a morphism, it follows that $\psi$
  is also a morphism.
\end{proof}

\subsection{Proof of Theorem~\ref{thm:reduction-to-S}(1)}
We prove part (1) of the Theorem by using the geometry of the 
displayed diagram with the key input being
Proposition~\ref{prop:geometry}.  For typographical convenience, if
$A$ is a finitely generated abelian group, we write $A[1/2d]$ for
$A\tensor\Z[1/2d]$. 

By the Shioda-Tate isomorphism (e.g., \cite[Ch.~4]{Ulmer14b}), we have
a direct sum decomposition
$$\NS(\EE)[1/2d]\cong E(\Fq(u))[1/2d]\oplus T[1/2d]$$
where $T$ is the subgroup of $\NS(\EE)$ generated by the zero section
and the irreducible components of the fibers.  Since $\WW$ is obtained
from $\EE$ by contracting all components of fibers not meeting the
zero section, we have 
$$\NS(\WW)[1/2d]\cong E(\Fq(u))[1/2d]\oplus\<O,F\>[1/2d]$$ 
where $O$ and $F$ are the classes of the zero section and a fiber of
$\WW\to\P^1$ respectively.  These decompositions are orthogonal for
the intersection pairings.  The fibration $\WW\to\P^1_u$ is the base
change of a fibration $\WW\to\P^1_t$, so $G$ acts on $\WW$ and
$\NS(\WW)$.  This action is trivial on $\<O,F\>$ and the last
displayed isomorphism is compatible with the $G$ actions.

Since $\tilde\SS$ is obtained from $\SS$ by blowing up eight points,
we have an orthogonal decomposition
$$\NS(\tilde\SS)\cong\Z^8\oplus\NS(\SS)\cong\Z^{10}\oplus\NS'(\SS).$$
The N\'eron-Severi group of the quotient $\tilde\SS/\Delta$ is
obtained by taking invariants, at least after inverting
$2d=|\Delta|$.  Noting that $\Delta$ permutes the
exceptional divisors of $\tilde\SS\to\SS$ in two orbits and that it
fixes the classes of $P\times\DD$ and $\CC\times Q$, we have
$$\NS(\tilde\SS/\Delta)[1/2d]\cong
\left(\NS(\tilde\SS)[1/2d]\right)^\Delta\cong
\Z[1/2d]^4\oplus\left(\vphantom{\tilde\SS}\NS'(\SS)[1/2d]\right)^\Delta.$$
The action of $G$ on $\CC$ induces an action on $\tilde\SS$ which
descends to $\tilde\SS/\Delta$.

Now we consider the morphism $\psi:\tilde\SS/\Delta\to\WW$ and use the
information provided by Proposition~\ref{prop:geometry}.  It is clear
from the coordinate expression for $\SS\ratto\WW$ that $\psi$ is
equivariant for the $G$ actions.  Part (2) tells us that the kernel of
$\NS(\tilde\SS/\Delta)\to\NS(\WW)$ has rank 2.  Parts (3) and (4)
allow us to determine it explicitly.

To that end, let $f_1$ and $f_2$ be the classes in $\NS(\tilde\SS)$ of
the curves $P\times\DD$ and $\CC\times Q$ respectively.  Also, let
$e_1$ and $e_2$ denotes the classes in $\NS(\tilde\SS)$ of the
exceptional divisors over $P_+\times Q'_+$ and $P'_+\times Q_+$
respectively.  Set $F_i=\pi_*f_i$ and $E_i=\pi_*e_i$ for $i=1,2$.
Then $E_1,E_2,F_1,F_2$ form a basis for the ``trivial part''
$\Z[1/2d]^4$ of $\NS(\tilde\SS/\Delta)[1/2d]$.

By part (3), $\psi_*F_1=\psi_*F_2=\phi_*f_1=\phi_*f_2=$ the class of a
bisection of $\WW\to\P^1$ with inverse points in each fiber.  This
class is easily seen to be $2O+dF$.  Similarly, part (4) tells us that
$\psi_*E_1=\phi_*e_1=O+(d/2)F$ (here we use that we have inverted 2),
and $\psi_*E_2=\phi_*e_2=O$.  The kernel of
$$\NS(\tilde\SS/\Delta)[1/2d]\to\NS(\WW)[1/2d]$$
is thus spanned by $F_1-F_2$ and $F_1-2E_1$.  Moreover, we have that
$\psi_*$ induces an isomorphism
$$\left(\NS'(\SS)[1/2d]\right)^\Delta\cong
\frac{\NS(\tilde\SS/\Delta)[1/2d]}{\<
  F_1,F_2,E_1,E_2\>}\cong
\frac{\NS(\WW)[1/2d]}{\<
  O,F\>}.$$
It follows that
$$\left(\NS'(\SS)[1/2d]\right)^\Delta\cong
E(\Fq(u))[1/2d]$$
and that this isomorphism is compatible with the height and
intersection pairings and the $G$ actions.

This completes the proof of part (1) of the Theorem.

\subsection{Proof of Theorem~\ref{thm:reduction-to-S}(2)}
Two fundamental results of Grothendieck (in \cite{Grothendieck68iii}, see also
\cite[5.3]{Ulmer14b}) say that the Tate-Shafarevich group of
$E/\Fq(u)$ and the Brauer group of $\EE$ are canonically isomorphic,
and that the Brauer group of a surface is a birational invariant.
Applying this to the diagram just before
Proposition~\ref{prop:geometry} shows that
$\sha(E/\Fq(u))\cong\Br(\tilde\SS/\Delta)$.  Since the order of
$\Delta$ is prime to $p$, we have
$$\Br(\tilde\SS/\Delta)[p^\infty]\cong\Br(\tilde\SS)[p^\infty]^\Delta
\cong\Br(\SS)[p^\infty]^\Delta.$$
This yields the isomorphism stated in part (2) of the theorem, and
this isomorphism is compatible with the $G$ actions because the maps
in the diagram above are $G$-equivariant.

\subsection{A higher genus generalization}\label{ss:Xr}
The results in this section generalize readily to a higher genus
example.  Specifically, fix an integer $r>1$ prime to $p$,
and let $X$ be the smooth, proper curve over $\Fp(t)$ defined by
$$y^r=x^{r-1}(x+1)(x+t).$$
The genus of $X$ is $r-1$.  We consider $X$ and its Jacobian $J=J_X$
over extensions $\Fq(u)$ where $u^d=t$, $d$ is prime to $p$, and $\Fq$
is a finite extension of $\Fp(\mu_d,\mu_r)$.  When $d=p^f+1$ and $r$
divides $d$, there are explicit divisors on $X$ yielding a
subgroup of $J(\Fq(u))$ of rank $(r-1)(d-2)$ and finite index.  This
situation is studied in detail in \cite{AIMgroup}.  

Let $\XX\to\P^1_u$ be the minimal regular model of $X$ over the
projective line whose function field is $\Fq(u)$.  Let $\CC=\DD$ be
the smooth, proper curve over $\Fq$ with equation 
$$z^d=x^r-1.$$
Then $\CC$ and $\DD$ carry actions of $\mu_r\times\mu_d$ and we let
$\Delta=\mu_r\times\mu_d$ act on $\SS=\CC\times_{\Fq}\DD$
``anti-diagonally.''  Arguments parallel to those in the proof of
Proposition~\ref{prop:geometry} show that $\XX$ is birationally
isomorphic to $\SS/\Delta$.  Using this, the arguments proving
Theorem~\ref{thm:reduction-to-S} generalize readily to give
isomorphisms
$$J(\Fq(u))[1/rd]\cong\NS'(\SS)[1/rd]^\Delta$$
and
$$\sha(J/\Fq(u))[p^\infty]\cong\Br(\SS)[p^\infty]^\Delta.$$

\section{Arithmetic of a product of curves}\label{s:arithmetic-of-CxD}
In this section, $k$ is a finite field of characteristic $p$, and
$\CC$ and $\DD$ are smooth, projective curves over $k$. Our goal is to
give a crystalline description of $\NS'(\CC\times\DD)$ and
$\Br(\CC\times\DD)$.  The former is due to Tate, and
the latter was done under somewhat restrictive hypotheses by Dummigan
(by a method he says was inspired by a letter of the author, see
\cite[p.~114]{Dummigan99}).  We use a variant of the method to give
the result in general.

\subsection{Flat and crystalline cohomology}
For the rest of this section, we write $W$ for the Witt-vectors
$W(k)$ and $\sigma$ for the Witt-vector Frobenius (lifting the
$p$-power Frobenius of $k$).

Given a smooth projective variety $\XX$ over $k$, we consider the crystalline
cohomology groups of $\XX$ and use the simplified notation
$$H^i(\XX):=H^i_{crys}(\XX/W)$$
for typographical convenience.  These groups are $W$-modules with a
$\sigma$-semi-linear action of the absolute Frobenius, denoted $F$.
When $\XX$ is a curve, we also define a $\sigma^{-1}$-semi-linear
action of Verschiebung, denoted $V$, on $H^1(\XX)$ by requiring that
$FV=VF=p$.  We write $A$ for the non-commutative ring $W\{F,V\}$
generated over $W$ by $F$ and $V$ with relations $Fa=\sigma(a)F$,
$aV=V\sigma(a)$, and $FV=VF=p$.

We will also consider cohomology of sheaves in the flat topology, say
the \emph{fppf} (faithfully flat, finitely presented) topology to fix
ideas.  Recall that
$H^1(\XX,\G_m)\cong\Pic(\XX)$ and that we define the Brauer group of
$\XX$ by
$$\Br(\XX):=H^2(\XX,\G_m).$$
If $\XX$ is smooth and $\dim\XX\le2$, it is known
\cite{Grothendieck68ii} that this definition agrees with that via
Azumaya algebras.

A well-known theorem of Weil asserts that $\CC$ and $\DD$ have
$k$-rational divisors of degree 1.  If $P$ and $Q$ are such, then the
classes in $\NS(\CC\times_k\DD)$ of $P\times\DD$ and $\CC\times Q$ are
independent of the choices of $P$ and $Q$.  We define
$\NS'(\CC\times_k\DD)$ as the orthogonal complement in
$\NS(\CC\times_k\DD)$ of these classes.

The goal of this section is to establish the following crystalline
calculations of the N\'eron-Severi and Brauer groups of a product of
curves.  

\begin{thm}\label{thm:cohom-of-product}\mbox{}
\begin{enumerate}
\item There is a functorial isomorphism
$$\NS'(\CC\times_k\DD)\tensor\Zp\isoto
\left(H^1(\CC)\tensor_{W} H^1(\DD)\right)^{F=p}.$$
\item There is a functorial exact sequence
$$0\to\left(\left(H^1(\CC)\tensor_W H^1(\DD)\right)^{F=p}\right)/p^n\to
\left(H^1(\CC)/p^n\tensor_W H^1(\DD)/p^n\right)^{F=V=p}
\to\Br(\CC\times_k\DD)_{p^n}\to0.$$
\end{enumerate}
Here the exponents mean the subgroups where $F$ and $V$ act as
indicated, and ``functorial'' means that the displayed maps are equivariant for
the action of $\aut(\CC)\times\aut(\DD)$.
\end{thm}

\begin{proof}
  We write $\XX$ for $\CC\times_k\DD$.  Part (1) is essentially the
  crystalline Tate conjecture.  More precisely, by a theorem of Tate
  (see \cite{WaterhouseMilne71}) we have an isomorphism
$$\NS(\XX)\tensor\Zp\cong H^2(\XX)^{F=p}.$$
Decomposing $\NS(\XX)$ as $\Z^2\oplus\NS'(\XX)$ and $H^2(\XX)$ via
the K\"unneth formula leads to the statement in part (1).

For part (2), we may assume that $\CC$ and $\DD$ have rational points.
Indeed, the theorem of Weil alluded to above shows that there is an
extension $k'/k$ of degree prime to $p$ such that $\CC$ and $\DD$ have
$k'$-rational points.  Using the Hochschild-Serre spectral sequences
in crystalline and flat cohomologies and the fact that taking
invariants under $\gal(k'/k)$ is an exact functor on groups of
$p$-power order shows that the theorem over $k'$ implies the theorem
over $k$.  We thus assume that $\CC$ and $\DD$ have $k$-rational
points.

Now consider the Kummer sequence
$$0\to\mu_{p^n}\to\G_m\to\G_m\to0$$
for the flat topology on $\XX$.  Taking flat cohomology yields
$$0\to\Pic(\XX)/p^n\to H^2(\XX,\mu_{p^n})\to \Br(\XX)_{p^n}\to 0.$$

Let $T=\Pic(\CC)/p^n\oplus\Pic(\DD)/p^n$.  The natural map
$T\to\Pic(\XX)/p^n$ is an injection with cokernel $\NS'(\XX)/p^n$.
Thus we have a commutative diagram with exact rows and columns:
$$\xymatrix{
&0\ar[d]&0\ar[d]\\
&T\ar[d]\ar@{=}[r]&T\ar[d]\\
0\ar[r]&\Pic(\XX)/p^n\ar[r]\ar[d]&H^2(\XX,\mu_{p^n})\ar[r]\ar[d]
&\Br(\XX)_{p^n}\ar[r]\ar@{=}[d]&0\\
0\ar[r]&\NS'(\XX)/p^n\ar[r]\ar[d]&H^2(\XX,\mu_{p^n})/T\ar[r]\ar[d]
&\Br(\XX)_{p^n}\ar[r]&0\\
&0&0}$$

Using part (1), we have 
$$\NS'(\XX)/p^n\cong
\left(\left(H^1(\CC)\tensor_{W(k)} H^1(\DD)\right)^{F=p}\right)/p^n,$$
so to complete the proof we must show that 
$$H^2(\XX,\mu_{p^n})/T\cong 
\left(H^1(\CC)/p^n\tensor H^1(\DD)/p^n\right)^{F=V=p}.$$

Let $\pi:\CC\times\DD\to\DD$ be the projection on the second factor.
We will compute $H^2(\XX,\mu_{p^n})$ via the Leray spectral sequence
for $\pi$.  By a theorem of Artin proven in \cite{Grothendieck68iii}, 
$$R^i\pi_*\G_m=\begin{cases}
\G_m&\text{if $i=0$}\\
\underline{\Pic}_{\XX/\DD}=\underline{\Pic}_{\CC/k}\times_k\DD&\text{if $i=1$}\\
0&\text{if $i>1$.}
\end{cases}$$
It follows that
$$R^i\pi_*\mu_{p^n}=\begin{cases}
\mu_{p^n}&\text{if $i=0$}\\
\underline{\Pic}_{\XX/\DD}[p^n]=J_\CC[p^n]&\text{if $i=1$}\\
\underline{\Pic}_{\XX/\DD}/p^n=\Z/p^n\Z&\text{if $i=2$}\\
0&\text{if $i>2$.}
\end{cases}$$ (Here we abuse notation slightly---the $k$-group schemes
on the right represent sheaves on $k$ and so by restriction sheaves on
$\DD$.)  Because $\CC$ has a rational point, $\pi$ has a section, so
the Leray spectral sequence degenerates at $E_2$ and we have that
$H^2(\XX,\mu_{p^n})$ is an extension of 
$$H^0(\DD,\Z/p^n\Z),\qquad H^1(\DD,J_\CC[p^n]),
\quad\text{and}\quad H^2(\DD,\mu_{p^n}).$$
The Kummer sequence on $\DD$ shows that 
$$H^2(\DD,\mu_{p^n})\cong\Pic(\DD)/p^n$$
which is an extension of $\Z/p^n\Z$ by $J_\DD(k)/p^n$.  Obviously
$H^0(\DD,\Z/p^n\Z)\cong \Z/p^n\Z$.  

To finish the proof, we must compute $H^1(\DD,J_\CC[p^n])$ in
crystalline terms.  First we make our notation a bit more precise.
Let $N$ be the sheaf on the flat site of $\spec k$ represented by the
finite flat group scheme $J_\CC[p^n]=\Pic_{\CC/k}[p^n]$.  Let
$\sigma$ be the structure map $\DD\to\spec k$ (which has a section
because $\DD$ has a rational point).  Then $H^1(\DD,J_\CC[p^n])$ means
$H^1(\DD,\sigma^*N)$.  Clearly, $\sigma_*\sigma^*N=N$.  By
\cite[III.4.16]{MilneEC} applied to $\sigma$, if $N'$ is the Cartier
dual of $N$, we have
$$R^1\sigma_*\sigma^* N\cong\underline{\Hom}_k(N',\underline\Pic_{\DD/k})
\cong\underline{\Hom}_k(N,\underline\Pic_{\DD/k}).$$
Here $\underline{\Hom}_k$ means the sheaf of homomorphisms of sheaves
on the flat site of $k$, and we have used that $\Pic_{\CC/k}[p^n]$ is
self-dual.

Now we consider the Leray spectral sequence for $\sigma$, which
degenerates because $\sigma$ has a section.  The sequence of low
degree terms is
$$0\to H^1(k,N)\to 
H^1(\DD,\sigma^*N)\to 
H^0(k,\underline{\Hom}_k(N,\underline\Pic_{\DD/k}))\to0.$$
Using 
$$0\to N\to J_{\CC}\labeledto{p^n}J_{\CC}\to0,$$ 
the equality of flat and \'etale cohomology for smooth group schemes,
and Lang's theorem (namely that $H^1(k,J_{\CC})=0$), we find that 
$H^1(k,N)=J_{\CC}(k)/p^n$.

Noting that the argument above applies with the roles of $\CC$
and $\DD$ reversed, we see that $\Pic(\CC)/p^n$ and $\Pic(\DD)/p^n$
are direct factors of $H^2(\XX,\mu_{p^n})$, and we find that 
$$H^2(\XX,\mu_{p^n})/T\cong
H^0(k,\underline{\Hom}_k(N,\underline\Pic_{\DD/k}))
\cong \Hom_k(J_{\CC}[p^n],J_{\DD}[p^n]).$$

We now turn to a crystalline description of the right hand group.
Letting $\mathbb{D}(\CC)$ and $\mathbb{D}(\DD)$ be the (contravariant)
Dieudonn\'e modules of the $p$-divisible groups of $J_\CC$ and $J_\DD$
respectively, the main theorem of Dieudonn\'e theory (equivalence of
categories) gives
$$\Hom(J_{\CC}[p^n],J_{\DD}[p^n])=
\Hom_A(\mathbb{D}(\DD)/p^n,\mathbb{D}(\CC)/p^n).$$ 
Here $\Hom_A$ means homomorphisms commuting with the action of
$A=W\{F,V\}$, i.e., with the actions of $F$ and $V$.

To finish, we use the result of Mazur and Messing
\cite{MazurMessingUEa1DCC} that $\mathbb{D}(\CC)\cong
H^1(\CC)$ and $\mathbb{D}(\DD)\cong H^1(\DD)$, and the duality
$\mathbb{D}(\DD)^*\cong\mathbb{D}(\DD)(-1)$ (Tate twist), so that
$$\Hom_A(\mathbb{D}(\DD)/p^n,\mathbb{D}(\CC)/p^n)\cong
\left(H^1(\CC)/p^n\tensor H^1(\DD)/p^n\right)^{F=V=p}.$$

This completes the proof of the theorem.
\end{proof}

\begin{rem}
A theorem of Illusie \cite[5.14]{Illusie79b} says that for a smooth
projective surface $\XX$ over an algebraically closed field $k$, we have
$$H^2(\XX,\Zp(1))\cong H^2(\XX/W(k))^{F=p}.$$

The proof of Theorem~\ref{thm:cohom-of-product}(2) can be adapted to
show that (when $\XX$ is a product of curves), this continues to hold
at finite level:  $H^2(\XX,\mu_{p^n})\cong H^2(\XX/W_n(k))^{F=p}$.
Conversely, a proof of this statement would yield a simple proof of
part (2) of the theorem (over an algebraically closed field).

On the other hand, the proof above shows that over a finite ground
field $H^2(\XX,\mu_{p^n})$ may be strictly bigger than
$H^2(\XX/W_n(k))^{F=p}$.  The point is that when $k$ is
algebraically closed, $\Pic(\CC)/p^n$ is $\Z/p^n\Z$ (because
$\Pic^0(\CC)$ is divisible), but it may be bigger when $k$ is finite.
\end{rem}

\section{Cohomology of $\CC$}\label{s:H1(C)}
In this section, we collect results on the crystalline cohomology of
the curve $\CC$ needed in the sequel.  Some of them may already be
available in the literature on Fermat curves, but for the convenience
of the reader we sketch arguments from first principles.

\subsection{Lifting}
From here until Subsection~\ref{ss:Cr}, $\CC$ will denote the smooth
projective model of the affine curve over $\Fp$ defined by
$z^d=x^2-1$.  (E.g., if $d$ is even, $\CC$ is the result of glueing
$\spec\Fp[x,z]/(z^d-x^2+1)$ and $\spec\Fp[x',z']/(z^{\prime
  d}-x^{\prime2}+1)$ via $(x',z')=(x/z^{d/2},1/z)$.  The case $d$ odd
is similar.)  The projective curve has a natural lifting to
$W(\Fp)=\Zp$ defined by the same equations.  We write $\CC/\Zp$ for
this lift.  It is smooth and projective over $\Zp$ with special fiber
$\CC$.

\subsection{Actions}
There is a canonical isomorphism $H^1_{crys}(\CC/\Zp)\cong
H^1_{dR}(\CC/\Zp)$ where the left hand side is the crystalline
cohomology of $\CC$ and the right hand side is the algebraic de Rham
cohomology of $\CC/\Zp$.  We will use this isomorphism to make the
crystalline cohomology explicit, endow it with a Hodge filtration, and
describe the actions of Frobenius, Verschiebung, $\mu_d$, and $\mu_2$ on it.  

Let $q$ be a power of $p$ congruent to 1 modulo $d$ so that $\Fq$
contains $\Fp(\mu_d)$.  Then $\CC/W(\Fq)=\CC/\Zp\times_{\Zp}W(\Fq)$
admits an action of the $d$-th roots of unity (acting on the
coordinate $z$) and $\mu_2=\pm1$ (acting on the coordinate $x$).

Recall that the absolute Frobenius of $\CC$ defines a $\Zp$-linear
homomorphism
$$F:H^1_{crys}(\CC/\Zp)\to H^1_{crys}(\CC/\Zp)$$ 
which induces a semi-linear homomorphism
$$F:H^1_{crys}(\CC/W(\Fq))\cong H^1_{crys}(\CC/\Zp)\tensor_{\Zp}W(\Fq)\to
H^1_{crys}(\CC/W(\Fq))$$ (semi-linear with respect to the Witt-vector
Frobenius $\sigma$).  We also have a $\sigma^{-1}$-semi-linear
endomorphism
$$V:H^1_{crys}(\CC/W)\to H^1_{crys}(\CC/W)$$
which is characterized by the formulas $FV=VF=p$.

Letting $\Fr_p\in\gal(\Fq/\Fp)$ act on
$\CC/\Fq=\CC\times_\Fp\Fq$ via the second factor, we get a semi-linear
endomorphism of $H^1(\CC/W)$ which fixes $H^1(\CC/\Zp)$.  Combining the
actions of $\mu_d$ and $\Fr_p$ gives a $\Zp$-linear action of
$G=\mu_d\sdp\gal(\Fq/\Fp)$ on $H^1_{crys}(\CC/W)$.

\subsection{A Basis}
By \cite[$0_{III}$, 12.4.7]{EGA3a}, we may define elements of
$H^1_{dR}(\CC/\Zp)$ by giving hypercocycles for an affine cover.  We
do so as follows:  For
$i=1,\dots,\lfloor (d-1)/2\rfloor$, let $e_i$ be the class defined by the
regular 1-form
$$\frac{z^{i-1}dz}{2x}.$$
Let $U_1$ be the affine curve defined by $z^d=x^2-1$ considered as a
Zariski open subset of $\CC/\Zp$.  Let $U_2$ be the complement of the
closed set where $z=0$ in $\CC/\Zp$.  Thus $U_1$ and $U_2$ define an
open cover of $\CC/\Zp$.  For $i=1,\dots,\lfloor (d-1)/2\rfloor$, the data
\begin{align*}
f^i_{12}&=\frac x{z^i}\in\OO_{\CC/\Zp}(U_1\cap U_2)\\
\omega^i_1&=\left(1-\frac{2i}{d}\right)\frac{dx}{z^i}\in\Omega^1_{\CC/\Zp}(U_1)\\
\omega^i_2&=\frac{ix\,dz}{z^{i+1}}-\frac{2i}{d}\frac{dx}{z^i}\in\Omega^1_{\CC/\Zp}(U_2)
\end{align*}
satisfies $df^i_{12}=\omega^i_1-\omega^i_2$ and so defines a class in
$H^1_{dR}(\CC/\Zp)$ which we denote $e_{d-i}$.  

\begin{props}\label{prop:CC-cohom}
  The classes $e_i$ \textup{(}$0< i<d$, $i\neq d/2$\textup{)} form a
  $\Zp$-basis of $H^1_{dR}(\CC/\Zp)$ and have the following
  properties:
\begin{enumerate}
\item The cup product $H^1_{dR}(\CC/\Zp)\times
  H^1_{dR}(\CC/\Zp)\to\Zp$ satisfies \textup{(}and is determined
  by\textup{)} the fact that for $0<i<d$ and $0<j<d$,
$$e_i\cup e_j=\begin{cases}
1&\text{if $i<d/2$ and $j=d-i$}\\
-1&\text{if $i>d/2$ and $j=d-i$}\\
0&\text{otherwise.}\end{cases}$$
\item The classes $e_i$ with $1\le i\le \lfloor (d-1)/2\rfloor$ form a
  $\Zp$-basis of the submodule $H^0(\CC/\Zp,\Omega^1_{\CC/\Zp})$ of
  $H^1_{dR}(\CC/\Zp)$, and the classes $e_i$ with
  $\lfloor(d+1)/2\rfloor\le i\le d-1$ project to a basis of the
  quotient module $H^1(\CC/\Zp,\OO_{\CC/\Zp})$.
\item The action of $\mu_d$ on $H^1_{crys}(\CC/W(\Fq))\cong
  H^1_{dR}(\CC/\Zp)\tensor_{\Zp}W(\Fq)$ is given by
$$[\zeta]e_i=\zeta^ie_i.$$
Also, $-1\in\mu_2$ acts on $H^1_{crys}(\CC/W(\Fq))$ as multiplication
by $-1$.
\item For $0<i<d$, we have $F(e_i)=c_ie_{pi}$ where
  $c_i\in\Zp$ satisfies
$$\ord(c_i)=\begin{cases}
0&\text{if $i>d/2$}\\
1&\text{if $i<d/2$.}\end{cases}$$
\textup{(}In $e_{pi}$, we read the subscript modulo $d$.\textup{)}
\item If $o\in O_{d,p}$, $d/\gcd(d,o)>2$, and $p$ is balanced modulo
  $d/\gcd(d,o)$ \textup{(}in the sense of
  Subsection~\ref{ss:balanced}\textup{)},
  then $\prod_{i\in o}c_i=\pm p^{|o|/2}$.  Equivalently, for all $i\in
  o$, $F^{|o|}e_i=\pm p^{|o|/2}e_i$.
\end{enumerate}
\end{props}

\begin{proof}
  Once we know that the $e_i$ form a basis, the formula in (1)
  determines the cup product.  To check the formula, one computes in
  the standard way:
  The cup product $e_i\cup e_{d-j}$ is given by the sum over
  points in $U_1$ of the residue of the meromorphic differential
  $z^{i-j}\,dz/(2z)$, and this sum is 1 or 0 depending on whether
  $j=i$ or not.

The formula in (1) implies that the classes $e_i$ with
$0<i<d$ and $i\neq d/2$ are linearly independent in
$H^1_{dR}(\CC/\Fp)$ and so they form an $\Fp$-basis since the genus of
$\CC$ is $(d-\gcd(d,2))/2$.  It follows that the $e_i$ form a
$\Zp$-basis of $H^1_{dR}(\CC/\Zp)$. 

It is clear from the definition that the $e_i$ with $0<i<d/2$ are
in the submodule $H^0(\CC/\Zp,\Omega^1_{\CC/\Zp})$ and so they form a
  basis by a dimension count.  Part (1) and Serre duality imply that
  the $e_i$ with $d/2<i<d$ project to a basis of
  $H^1(\CC/\Zp,\OO_{\CC/\Zp})$.  This proves part (2).

Part (3) follows immediately the definition of the $e_i$.  

It follows from part (3) that $F(e_i)=c_ie_{pi}$ for some
$c_i\in\Zp$.  Indeed, Frobenius must send the subspace of
$H^1(\CC/W(k))$ where $[\zeta]$ acts by $\zeta^i$ to the subspace
where it acts by $\zeta^{pi}$.  By (3), these subspaces are spanned by
$e_i$ and $e_{pi}$ respectively, so $F(e_i)=c_ie_{pi}$, and $c_i$
must lie in $\Zp$ since $F$ acts on $H^1_{crys}(\CC/\Zp)$.  The
assertion on the valuation of $c_i$ follows from \cite[Lemma,
p.~665]{Mazur72b} and \cite[top of p.~65]{Mazur73} This proves part
(4).

For part (5), a standard calculation
(cf.~\cite[Ch.~11]{IrelandRosenCIMNT}) gives the eigenvalues of
$F^{|o|}$ in terms of Jacobi sums.  Using the notation of
\cite{ConceicaoHallUlmer14},
$F^{|o|}e_i=\lambda(-1)J(\lambda,\chi^i)e_i$ where $\lambda$ is a
character of $k=\F_{p^{|o|}}$ of order $2$ and $\chi$ is a character
of order $d$.  By \cite[Prop.~4.1]{ConceicaoHallUlmer14}, the Jacobi
sum is $\pm p^{|o|/2}$.

This completes the proof of the Proposition.
\end{proof}

\begin{rem}\label{rem:u-l-motivation}
Part (4) of the Proposition is the reason for the minus signs in the
definition of the word attached to an orbit in
Subsection~\ref{ss:invs-of-orbits}.  Indeed, if $i<d/2$, so that $e_i$
is in $H^0(\CC/\Zp,\Omega^1_{\CC/\Zp})$, then $F(e_i)$ is divisible by
$p$ (i.e., its ``valuation'' has gone up), whereas if $i>d/2$, then
$F(e_i)$ is not divisible by $p$ (i.e., its ``valuation'' is still low).
\end{rem}

\subsection{Generalization to $r>2$}\label{ss:Cr}
Most of the above extends to the curve $\CC_r$ defined by $z^d=x^r-1$
for any $r$ which is $>1$ and relatively prime to $p$.  We give the
main statements; their proofs are entirely parallel to those in the
case $r=2$.

The curve $\CC_r$ has an obvious lift to $\Zp$ which we denote $\CC_r/\Zp$.
This yields an identification $H^1_{crys}(\CC_r/\Zp)\cong
H^1_{dR}(\CC_r/\Zp)$. 

For $i\in\Z/d\Z$, we write $\<i/d\>$ for the fractional part of $i/d$
(for any representative of the class of $i$).  Similarly for
$j\in\Z/r\Z$ and $\<j/r\>$.  Let $A$ be the subset of
$\Z/d\Z\times\Z/r\Z$ consisting of $(i,j)$ where $i\neq0$, $j\neq0$
and $\<i/d\>+\<j/r\>>1$.  Let $B$ be the subset where $i\neq0$,
$j\neq0$, and $\<i/d\>+\<j/r\><1$.  Let $S=A\cup B$.

There is a $\Zp$-basis of $H^1_{dR}(\CC_r/\Zp)$ consisting of classes
$e_{i,j}$ with $(i,j)\in S$ with the following properties:
\begin{enumerate}
\item $$e_{i,j}\cup e_{i',j'}=\pm\delta_{ii'}\delta_{jj'}$$
where the sign is $+$ if $(i,j)\in A$ and $-$ if $(i,j)\in B$.
\item The $e_{i,j}$ with $(i,j)\in A$ form a basis of
  $H^0(\CC_r/\Zp,\Omega^1_{\CC_r/\Zp})$, and the $e_{i,j}$ with
  $(i,j)\in B$ project to a basis of $H^1(\CC_r/\Zp,\OO_{\CC_r/\Zp})$.
\item If $q$ is such that $\Fq$ contains $\Fp(\mu_d,\mu_r)$, then the
  action of $\mu_d\times\mu_r$ on $H^1(\CC_r/W(\Fq))$ is given by
$$[\zeta_d,\zeta_r]e_{i,j}=\zeta_d^i\zeta_r^j e_{i,j}.$$
\item $F(e_{i,j})=c_{i,j}e_{pi,pj}$ where $c_{i,j}\in\Zp$ satisfies
$$\ord_p(c_{i,j})=\begin{cases}
0&\text{if $(i,j)\in B$}\\
1&\text{if $(i,j)\in A$.}\end{cases}$$
\end{enumerate}

There is also a notion of balanced which we now explain.  Let
$H=(\Z/\lcm(d,r)\Z)^\times$ and let $H$ act on $S$ by multiplication
in both coordinates.  Let $\<p\>$ be the cyclic subgroup of $H$
generated by $p$.  If $(i,j)\in S$, we say \emph{the ray through
  $(i,j)$ is balanced} if for all $t\in H$, the orbit $\<p\>t(i,j)$ is
evenly divided between $A$ and $B$, i.e., 
$$\left|\<p\>t(i,j)\cap A\right|=\left|\<p\>t(i,j)\cap B\right|.$$

The final property of $\CC_r$ we mention is:
\begin{itemize}
\item[(5)] For $(i,j)\in S$, let $o=\<p\>(i,j)$ and set
$$J_o=\prod_{(i',j')\in o}c_{i',j'}.$$
Then $J_o$ is a root of unity times $p^{|o|/2}$ if and only if the ray
through $(i,j)$ is balanced.
\end{itemize}

To prove this, we note that the displayed product is an eigenvalue of
$F^{|o|}$ on $H^1_{crys}(\CC_r/\Zp)$.  This eigenvalue may be
identified with a Jacobi sum, and arguments parallel to those in
\cite[Prop.~4.1]{ConceicaoHallUlmer14} using Stickelberger's Theorem
show that the Jacobi sum is a root of unity times $p^{|o|}$ if and
only if the ray through $(i,j)$ is balanced.  In
\cite{ConceicaoHallUlmer14}, these roots of unity were always $\pm1$.
If $r$ divides $d$ and $d$ divides $p^f+1$, then again these root of
unity are $\pm1$.  In the more general context, all we can say is that
they are roots of unity of order at most $\gcd(\lcm(r,d),p-1)$.

To close this section, we note that the apparatus of orbits, words, and
the associated invariants (as in Section~\ref{s:orbits}) applies as
well to the cohomology of $\CC_r$ as soon as we replace ``$i>d/2$''
and ``$i<d/2$'' with ``$(i,j)\in A$'' and ``$(i,j)\in B$"
respectively.

\section{$p$-adic exercises}\label{s:exercises}
Fix as usual an odd prime number $p$, a positive integer $d$
relatively prime to $p$, and an extension $\Fq$ of $\Fp(\mu_d)$, and
consider $E$ over $\Fq(u)$ where $u^d=t$.

Theorem~\ref{thm:reduction-to-S} and
Theorem~\ref{thm:cohom-of-product} reduce the problem of computing
$E(\Fq(u))$ and $\sha(E/\Fq(u))$ to exercises in semi-linear algebra with
raw data supplied by Proposition~\ref{prop:CC-cohom}.  

In this section, we carry out these $p$-adic exercises.

\subsection{Setup}
We write $W$ for the Witt vectors $W(\Fq)$, $W_n$ for $W_n(\Fq)$,
$H^1(\CC)$ for $H^1_{crys}(\CC/W)$, and $H^1(\DD)$ for
$H^1_{crys}(\DD/W)$ where $\CC=\DD$ is the curve over $\Fq$ studied in
Section~\ref{s:H1(C)}.  The product $\CC\times_\Fq\DD$ carries an
action of $\Delta=\mu_2\times\mu_d$ acting ``anti-diagonally'' as well
as an action of $G=\mu_d\sdp\gal(\Fq/\Fp)$ acting on the factor $\CC$.

Our goal is to compute
$$H:=\left(H^1(\CC)\tensor_W H^1(\DD)\right)^{\Delta, F=V=p}$$
and
$$H_n:=\left(H^1(\CC/W_n)\tensor_W H^1(\DD/W_n)\right)^{\Delta, F=V=p}.$$

For an orbit $o\in O_{d,p}$, we write $H^o$ and $H^o_n$ for the $o$
parts of the corresponding groups, i.e., for the images of the
projector $\pi_o$ on $H$ or $H_n$.

Since $H^1(\CC)$ and $H^1(\DD)$ free $W$-modules and the order of $\Delta$ is
prime to $p$, we have
\begin{align*}
\left(H^1(\CC/W_n)\tensor_W H^1(\DD/W_n)\right)^\Delta
&=\left(\left(H^1(\CC)\tensor_W H^1(\DD)\right)/p^n\right)^\Delta\\
&=\left(\left(H^1(\CC)\tensor_W H^1(\DD)\right)^\Delta\right)/p^n
\end{align*}
so the first step in both cases is to compute 
$M=\left(H^1(\CC)\tensor_W
  H^1(\DD)\right)^\Delta$.

\subsection{A basis for $M$}
By Proposition~\ref{prop:CC-cohom}(3), $\mu_2$ acts as $-1$ on
$H^1(\CC)$ and $\mu_d$ acts on $e_i$ by $\chi^i$.  Thus $\mu_2$ acts
trivially on $H^1(\CC)\tensor_W H^1(\DD)$ and $\mu_d$ acts on
$e_i\tensor e_j$ by $\chi^{i-j}$.  Therefore we have
$$M\cong\bigoplus_{i\in\Z/d\Z\setminus\{0,d/2\}}W(e_i\tensor e_i).$$
We decompose $M=\oplus_{o\in O}M^o$ where
$$M^o=\bigoplus_{i\in o}W(e_i\tensor e_i).$$

For the rest of this section, we fix an orbit $o$ and \emph{we assume
  that $\gcd(o,d)<d/2$ and $p$ is balanced modulo $d/\gcd(o,d)$}.
By Theorem~\ref{thm:disc}, this is the situation in which
$E(\Fq(u)\tensor\Zp)^o\neq0$, and it turns out to be the
situation in which we can say something non-trivial about
$\sha(E/\Fq(u))^o$.

As a first step, we make a change of basis which is perhaps unnatural,
but has the virtue of simplifying the notation considerably.  Namely,
let $i\in o$ be the standard base point (see
Definition~\ref{def:base-pts}), and let
$$d_{ip^j}=\begin{cases}
c_{ip^j}&\text{if $w_j=l$}\\
c_{ip^j}/p&\text{ if $w_j=u$}
\end{cases}$$
where the $p$-adic integers $c_{ip^j}$ are defined in
Proposition~\ref{prop:CC-cohom}(4).  That proposition implies 
that the $d_{ip^j}$ are units.
Set $f_i=e_i\tensor e_i$, and for
$j=1,\dots,|o|-1$, set
$$f_{ip^j}=\left(\prod_{\ell=1}^jd_{pi^\ell}^{2}\right)e_{ip^j}\tensor e_{ip^j}.$$
Then $\{f_j|j\in o\}$ forms a $W$-basis of $M^o$, and it
follows from Proposition~\ref{prop:CC-cohom} parts (4) and (5) that for all
$j\in o$ we have
$$F(f_j)=\begin{cases}
p^2 f_{pj}&\text{if $j<d/2$}\\
f_{pi}&\text{if $j>d/2$.}
\end{cases}
$$
(Here as usual, we read the subscripts modulo $d$.)

Similarly, we have 
$$V(f_j)=\begin{cases}
f_{p^{-1}j}&\text{if $p^{-1}j<d/2$}\\
p^2 f_{p^{-1}j}&\text{if $p^{-1}j>d/2$}
\end{cases}
$$
where ``$p^{-1}j<d/2$'' means that the least positive residue of
$p^{-1}j$ is $<d/2$.

We have a remaining action of $G=\mu_d\sdp\gal(\Fq/\Fp)$ on $M$ via
its action on the first factor in $H^1(\CC)\tensor_W H^1(\DD)$.  Under
this action, $\zeta\in\mu_d$ acts $W$-linearly as
$[\zeta]f_j=\zeta^jf_j$ and $\Fr_p\in\gal(\Fq/\Fp)$ acts semi-linearly
as $\Fr_p(\alpha f_j)=\sigma(\alpha)f_j$.

\subsection{mod $p$ case with $d=p^f+1$ and
  $\Fq=\Fp(\mu_d)$}\label{ss:modp} 
As a very easy first case, we assume $d=p^f+1$ and $\Fq=\Fp(\mu_d)$,
and we compute $H_1$, which is just the subspace of $M/p$ killed by
$F$ and by $V$.  We saw just above that $F(f_i)$ is zero if and only
if $i<d/2$, i.e., if and only if the first letter in the word
associated to $i$ is $u$.  Similarly, $V(f_i)=0$ if and only if the
last letter of the word of $i$ is $l$.  This yields the first part of
the following statement.

\begin{prop}\label{prop:H1-via-crys}
 If $d=p^f+1$ and $\Fq=\Fp(\mu_d)$, then
$$H_1:=\left(H^1(\CC/\Fq)\tensor_W H^1(\DD/\Fq)\right)^{\Delta,F=V=0}$$
is spanned over $\Fq$ by the classes $f_i$ where the word of $i$ has
the form $u\cdots l$.  If the first half of the word of $o$ has the
form $u^{e_1}l^{e_2}\cdots u^{e_k}$ with each $e_i>0$, then the
$\Fq$-dimension of $H^o_1$ is $k$.  We have
$$\dim_{\Fq}H_1=\pfrac{p-1}2\pfrac{p^{f-1}+1}2.$$
\end{prop}

The dimension counts in the proposition will be proven at the end of
Subsection~\ref{ss:counting} after we have proven
Lemma~\ref{lemma:counting}.

\subsection{The basic equations}
We now make first reductions toward computing $H^o$ and $H^o_n$ in
general.  Focus on one orbit $o\in O$ with its standard base point $i$
and associated word $w=w_1\cdots w_{|o|}$.

Consider a typical element $c\in M^o$ (or in $M_n^o$):
$$c=\sum_{j=0}^{|o|-1}\alpha_jf_{ip^j}$$
where $\alpha_j\in W$ (or in $W_n$), and where we read the index $j$
modulo $|o|$.

Then the class $c$ satisfies $(F-p)(c)=0$ if and only if
$$p\alpha_{j+1}=\begin{cases}
\sigma(\alpha_j)&\text{if $w_j=l$}\\
p^2\sigma(\alpha_j)&\text{if $w_j=u$}
\end{cases}$$
for $j=0,\dots,|o|-1$.
Similarly, the class $c$ satisfies $(V-p)(c)=0$ if and only if 
$$p\alpha_j=\begin{cases}
p^2\sigma^{-1}(\alpha_{j+1})&\text{if $w_j=l$}\\
\sigma^{-1}(\alpha_{j+1})&\text{if $w_j=u$}
\end{cases}$$
for $j=0,\dots,|o|-1$.

Note that when $w_j=l$, the equation coming from $V-p=0$ follows from
that coming from $F-p=0$, and when $w_j=u$, then the equation coming
from $F-p=0$ follows from that coming from $V-p=0$.  Thus $c$
satisfies $(F-p)(c)=(V-p)(c)=0$ if and only if
\begin{equation}\begin{cases}
\alpha_j=\sigma^{-1}p\alpha_{j+1}&\text{if $w_j=l$}\\
\sigma p\alpha_j=\alpha_{j+1}&\text{if $w_j=u$}
\end{cases}
\end{equation}
for $j=0,\dots,|o|-1$.


Note that when $w_j=l$, $\alpha_{j+1}$ determines $\alpha_j$, and when
$w_j=u$, $\alpha_j$ determines $\alpha_{j+1}$.  Thus we may eliminate
many of the variables $\alpha_j$.
More precisely, write the word $w$ in exponential form:
$w=u^{e_1}l^{e_2}\cdots l^{e_{2k}}$.  Setting $\beta_0=\alpha_0$ and
$$\beta_j=\alpha_{e_1+e_2+\cdots+e_{2j}}$$
for $1\le j\le k$ (so that $\beta_k=\beta_0$), the class $c$ is entirely
determined by the $\beta$'s.  Indeed, for $\sum_{i=1}^{2j}e_i\le \ell\le
\sum_{i=1}^{2j+1}e_i$, we have
$$\alpha_\ell=(\sigma p)^{\ell-\sum_{i=1}^{2j}e_i}\beta_j$$
and for $\sum_{i=1}^{2j+1}e_i\le\ell\le\sum_{i=1}^{2j+2}e_i$, we have
$$\alpha_\ell=(\sigma^{-1}p)^{\sum_{i=1}^{2j+2}e_i-\ell}\beta_{j+1}.$$

The conditions on the $\alpha$'s translated to the $\beta$'s become
\begin{align}\label{eq:basic}
(\sigma p)^{e_1}\beta_0&=(\sigma^{-1}p)^{e_2}\beta_1\notag\\
(\sigma p)^{e_3}\beta_1&=(\sigma^{-1}p)^{e_4}\beta_2\notag\\
&\vdots\\
(\sigma p)^{e_{2k-1}}\beta_{k-1}&=(\sigma^{-1}p)^{e_{2k}}\beta_k\notag
\end{align}

We refer to these as the \emph{basic equations}.  

The upshot is that the coordinates $\beta$ define an embedding
$H^o\into W^k$ (resp.~$H^o_n\into W_n^k$) with $c\mapsto
(\beta_j)_{j=1,\dots,k}$ whose image is characterized by the basic
equations.  

In the rest of this section, we will make this image more
explicit in the ``adic case'' $H^o\into W^k$ and the ``mod $p^n$
case'' $H^o_n\into W_n^k$.

\subsection{adic case}
In this case, the  $\beta_j$ lie in $W$ which is torsion free,  so
the basic equations allow us to eliminate all $\beta_j$ with $0<j<k$
in favor of $\beta_0$.  Indeed, the basic equations imply that
\begin{align}\label{eq:eliminate}
\beta_1&=\sigma^{e_1+e_2}p^{e_{1,2}}\beta_0\notag\\
\beta_2&=\sigma^{e_3+e_4}p^{e_3-e_4}\beta_1
     =\sigma^{e_1+\cdots+e_4}p^{e_{1,4}}\beta_0\notag\\
&\vdots\\
\beta_k&=\sigma^{e_1+\cdots+e_{2k}}p^{e_{1,2k}}\beta_0
     =\sigma^{|o|}p^{e_{1,2k}}\beta_0=\sigma^{|o|}\beta_0\notag
\end{align}
where as usual, $e_{ij}$ denotes the alternating sum
$$e_{ij}=e_i-e_{i+1}+\cdots\pm e_j.$$
 
Note that $\beta_k=\beta_0$, so the last equation is satisfied if and
only if $\beta_0\in W(\F_{p^{|o|}})$.  Note also that since $i$ is a
good base point, the $e_{1j}$ are
$\ge0$ for $1\le j\le 2k$, so the exponents of $p$ on the far
right hand sides of the equations above are non-negative.  Therefore,
for any choice of $\beta_0\in W(\F_{p^{|o|}})$, the equations give
well-defined elements $\beta_j\in W(\F_{p^{|o|}})\subset W$ solving
the basic equations.

The upshot is that the map sending $c\mapsto\beta_0=\alpha_0$ gives an
isomorphism $H^o\cong W(\F_{p^{|o|}})=\Gamma_o$.  The inverse of this
map is
\begin{equation*}
\alpha_0\mapsto \sum_{j=0}^{|o|-1}\sigma^jp^{a_j}\alpha_0f_{ip^j}
\end{equation*}
where $a_j$ is the function defined in
Subsection~\ref{ss:invs-of-orbits}.  It is easy to see that this map
is equivariant for the action of $G=\mu_d\sdp\gal(\Fq/\Fp)$ where $G$
acts on $W(\F_{p^{|o|}})\cong\Gamma_o$ as in
Proposition~\ref{prop:G-reps}. 

Summing up:
\begin{prop}\label{prop:H^o}  
  Suppose that $o\in O_{d,p}$ is an orbit with $\gcd(d,o)<d/2$ and $p$
  is balanced modulo $p$.  Then the map above induces an isomorphism
  of $\Zp[G]$-modules
$$H^o\cong\Gamma_o.$$
\end{prop}

\subsection{mod $p^n$ case}
To compute $H^o_n$, we should solve the basic equations
\eqref{eq:basic} with the $\beta_j\in W_n$.  We will do this for all
sufficiently large $n$ (to be made precise just below).  We write
$\beta_j^{(\nu)}$ for the Witt-vector components of $\beta_j$, and by
convention we set $\beta_j^{(\nu)}=0$ if $\nu\le0$.

Recall that the \emph{height} of an orbit with word
$u^{e_1}l^{e_2}\cdots l^{e_{2k}}$ is
$$ht(o)=\max\{e_1, e_{13}, \dots,e_{1,2k-1}\}.$$
In other words, $ht(o)$ is the maximum value of the sequence $a_j$
associated to $o$ in Subsection~\ref{ss:more-invs}.  For the rest of
this section \emph{we assume that $n\ge ht(o)$}.

Taking the $\nu$-th Witt component in the basic equations
\eqref{eq:basic} yields the following system of equations in $\Fq$:
\begin{align}\label{eq:basic-n}
\sigma^{2e_1}\beta_0^{(\nu-e_1)}&=\beta_1^{(\nu-e_2)}\notag\\
\sigma^{2e_3}\beta_1^{(\nu-e_3)}&=\beta_2^{(\nu-e_4)}\notag\\
&\vdots\\
\sigma^{2e_{2k-1}}\beta_{k-1}^{(\nu-e_{2k-1})}&=\beta_k^{(\nu-e_{2k})}.\notag
\end{align}
Now suppose that $\nu\le n-ht(o)$, so that $\nu+e_1\le n$,
$\nu+e_{13}\le n$, etc.  Considering the $\nu+e_1$ component of the
first equation in display~\eqref{eq:basic-n}, the $\nu+e_{13}$
component of the second equation, etc., leads to the chain of
equalities
$$\beta_0^{(\nu)}=\sigma^{-2e_1}\beta_1^{(\nu-e_{12})}=
\sigma^{-2(e_1+e_3)}\beta_2^{(\nu-e_{14})}=\cdots=
\sigma^{-2(e_1+e_3+\cdots+e_{2k-1})}\beta_0^{(\nu-e_{1,2k})}
=\sigma^{-|o|}\beta_0^{(\nu)}.$$
It follows that for $\nu\le n-ht(o)$, $\beta_0^{(\nu)}$ lies in
$\F_{p^{|o|}}$.  

Conversely, given Witt components $\beta_0^{(\nu)}\in\F_{p^{|o|}}$ for
$\nu\le n-ht(o)$, there exists a solution
$(\beta_0,\dots,\beta_{k-1})\in W_n^k$ of the basic equations with the
given components.  Indeed, we may complete $\beta_0$ to an element of
$W$, use the equations~\eqref{eq:eliminate} to define the
other $\beta_j$, and then reduce modulo $p^n$.

Thus the map $(\beta_0,\dots,\beta_{k-1})\mapsto
\beta_0\pmod{p^{n-ht(o)}}$ defines a surjective homomorphism 
\begin{equation}\label{eq:reduction}
H^o_n\to W_{n-ht(o)}(\F_{p^{|o|}})
\end{equation}
whose kernel is easily seen to be $p^{n-ht(o)}H^o_n$.  Note that if
$n_2\ge n_1\ge ht(o)$, we have an isomorphism
$$p^{n_1-ht(o)}H^o_{n_1}\cong p^{n_2-ht(o)}H^o_{n_2}$$
which sends $(\beta_j)$ to $p^{n_2-n_1}(\beta_j)$.  In this sense, the
kernel of the surjection~\eqref{eq:reduction} is independent of $n$
(as long as $n\ge ht(o)$).  Thus to compute it, we may assume that
$n=ht(o)$ and compute $H^o_{ht(o)}$.

Next we note that if $(\beta_j)\in H^o_{ht(o)}$ and if $\ell$ is such
that $ht(o)=e_{1,2\ell+1}=e_{2\ell+2,2k}$, then 
\begin{align*}
0=p^{ht(o)}\beta_k&=p^{e_{2\ell+2,2k}}\beta_k\\
&=p^{e_{2\ell+2,2k-2}}\beta_{k-1}\\
&\vdots\\
&=p^{e_{2\ell+2}}\beta_{\ell+1}
\end{align*}
Thus after reordering we may write the basic equations as a triangular system:
\begin{align*}
(\sigma p)^{e_{2\ell+3}}\beta_{\ell+1}&=(\sigma^{-1}p)^{e_{2\ell+4}}\beta_{\ell+2}\\
&\vdots\\
(\sigma p)^{e_{2k-1}}\beta_{k-1}&=(\sigma^{-1}p)^{e_{2k}}\beta_{k}\\
(\sigma p)^{e_{1}}\beta_{k}&=(\sigma^{-1}p)^{e_{2}}\beta_{1}\\
&\vdots\\
(\sigma p)^{e_{2\ell-1}}\beta_{\ell-1}&=(\sigma^{-1}p)^{e_{2\ell}}\beta_\ell\\
(\sigma p)^{e_{2\ell+1}}\beta_{\ell}&=0.
\end{align*}

Now introduce new variables $\gamma_j$ indexed by $j\in \Z/k\Z$ and
related to the $\beta_j$ by
$$\gamma_{j-\ell}=
\begin{cases}
\sigma^{-e_1-e_2-\cdots-e_{2j}}\beta_j&\text{if $1\le j\le \ell$}\\
\sigma^{e_{2j+1}+e_{2j+2}+\cdots+e_{2k}}\beta_j&\text{if $\ell+1\le j\le k.$}
\end{cases}$$

In these variables, the basic equations become
\begin{align*}
p^{e_{2\ell+3}}\gamma_{1}&=p^{e_{2\ell+4}}\gamma_{2}\\
&\vdots\\
p^{e_{2k-1}}\gamma_{k-\ell-1}&=p^{e_{2k}}\gamma_{k-\ell}\\
p^{e_{1}}\gamma_{k-\ell}&=p^{e_{2}}\gamma_{k+1-\ell}\\
&\vdots\\
p^{e_{2\ell-1}}\gamma_{k-1}&=p^{e_{2\ell}}\gamma_k\\
p^{e_{2\ell+1}}\gamma_{k}&=0.
\end{align*}
or, in matrix form
$$B(e_{2\ell+3},\dots,e_{2k},e_1,\dots,e_{2\ell+1})
\begin{pmatrix}
\gamma_{1}\\ \gamma_{2}\\ \vdots\\ \gamma_k
\end{pmatrix}=0.$$

The upshot is that we have identified $H^o_{ht(o)}$ with the kernel of
$B(e_{2\ell+3},\dots,e_{2\ell+1})$ on $W_{ht(o)}^k$.  By
Remark~\ref{rem:equivalent-Bs}, this is the same as the kernel of
$B(e_1,\dots,e_{2k+1})$, and this kernel is described by the invariant
factors $d_j$ analyzed in Subsection~\ref{ss:computing-invariants}.

To finish the discussion, we will unwind the action of
$G=\mu_d\sdp\gal(\Fq/\Fp)$ under the isomorphisms above.  The action
of $\Fr_p$ on a class $c\in H^o_n$ goes over to the action of $\sigma$
on the coordinates $\alpha_j$ and also on the coordinates $\beta_j$
and $\gamma_j$.  The action of $\zeta\in\mu_d$ on $c$ goes over to
multiplication by $\zeta^{ip^j}$ on $\alpha_j$, so to multiplication
by $\zeta^{ip^{e_1+e_2+\cdots+e_{2j}}}$ on $\beta_j$, and finally to
multiplication by $\zeta^i$ on the $\gamma_j$.

The following statement summarizes the results of this subsection.

\begin{prop}\label{prop:H^o_n}
  Suppose that $o\in O_{d,p}$ is an orbit with $\gcd(d,o)<d/2$ and $p$
  is balanced modulo $p$.  Suppose that the word of $o$ is
  $u^{e_1}\cdots l^{e_{2k}}$ and recall the invariants $d_1,\dots,d_k$
    attached to $o$ in Subsection~\ref{ss:more-invs}.
 \begin{enumerate}
 \item For all $n\ge ht(o)$ we have an exact sequence of
   $\Zp[G]$-modules
$$0\to\bigoplus_{j=1}^kW_{d_j}(\Fq)\to H^o_n\to
W_{n-ht(o)}(\F_{p^{|o|}})\to0.$$
Here $G$ acts on the Witt vectors as described in
Proposition~\ref{prop:G-reps}(5).  
\item The cokernel of $H^o/p^n\to H^o_n$ is isomorphic to
$$\frac{\bigoplus_{j=1}^kW_{d_j}(\Fq)}{W_{d_k}(\F_{p^{|o|}})}.$$
\end{enumerate}
\end{prop}

The first part was proven earlier in this subsection.  The second
follows from the fact that the composed map $H^o/p^n\to H^o_n\to
W_{n-ht(o)}(\F_{p^{|o|}})$ (cf.~equation~\eqref{eq:reduction}) is
obviously surjective, with kernel $p^{n-ht(o)}H^o/p^nH^o$ and
$d_k=ht(o)$.

\begin{rem}
The ``d\'evissage'' implicit in this subsection is captured by the
middle column of the following diagram with exact rows and columns:
$$\xymatrix{
&0\ar[d]&0\ar[d]\\
0\ar[r]&\frac{p^{n-ht(o)}H^o}{p^nH^o}\ar[r]\ar[d]&p^{n-ht(o)}H^o_n\ar[r]\ar[d]
&\Br^o[p^n]\ar[r]\ar@{=}[d]&0\\
0\ar[r]&\frac{H^o}{p^nH^o}\ar[r]\ar[d]&H^o_n\ar[r]\ar[d]
&\Br^o[p^n]\ar[r]&0\rlap{.}\\
&\frac{H^o}{p^{n-ht(o)}H^o}\ar[d]\ar[r]^{\cong}&\frac{H^o_n}{p^{n-ht(o)}H^0_n}\ar[d]\\
&0&0}$$
Here $\Br^o[p^n]$ is the $p^n$-torsion in $\Br(\CC\times_\Fq\DD)^{\Delta,o}$
and the middle row is the $o$ part of the exact sequence in
Theorem~\ref{thm:cohom-of-product}(2).  The middle column is the $o$
part of the exact sequence of \cite{Artin74} on page 553 just after
equation (3.2) and \cite[p.~521, line 6]{Milne75}, i.e.,
$U^2(p^\infty)=U^2(p^{ht(o)})=p^{n-ht(o)}H^o_n$ and
$D^2(p^{n-ht(o)})=H^o_n/p^{n-ht(o)}H^o_n$.  Note also that the top row
above shows that $\Br$ is not represented by an algebraic group, even
as a functor on finite fields.
\end{rem}

\section{Proofs of the main results}\label{s:proofs}
In this section we prove an easy lemma on counting words then
assemble the results from Sections~\ref{s:DPC},
\ref{s:arithmetic-of-CxD}, and \ref{s:exercises} to prove the theorems
stated in Sections~\ref{s:intro} and \ref{s:refined-results}.

\subsection{Counting patterns}\label{ss:counting}
Let $f$ be a positive integer, let $d=p^f+1$, and let
$S=\Z/d\Z\setminus\{0,d/2\}$.  Let $\<p\>\subset(\Z/d\Z)^\times$ be
the cyclic subgroup generated by $p$.  Given $i\in S$ we define a
string $w$ of length $f$ in the alphabet $\{u,l\}$, called the
\emph{pattern} associated to $i$, as follows: $w=w_1\cdots w_f$ where
$$w_j=\begin{cases}
l&\text{if $-p^{j-1}i\in A$}\\
u&\text{if $-p^{j-1}i\in B$}.
\end{cases}$$ 
If the orbit $o$ of $\<p\>$ through $i$ has full size (i.e., size
$2f$), then the pattern of $i$ is the same thing as the first half of
the word associated to $i$.  If the orbit is smaller, then the pattern
is a repetition of the $\lfloor f/|o|\rfloor$ copies of the word,
followed by the first half of the word.  (Note that $f/|o|$ always has
denominator 2 because the second half of the word is the complement of
the first.)
For example, if $p=f=3$ and $i=7$, then $o=\{7,21\}$, the associated
word is $ul$, and the pattern is $ulu$.  Patterns turn out to be more
convenient than words for counting.

Let $T$ be the set of tuples
$$T=\left\{\left.(i_1,\dots,i_f)\,\right|\, i_j\in\{0,\dots,p-1\},\
  \text{not all $i_j=(p-1)/2$}\right\}.$$
There is a bijection $T\to S$ which sends 
$$(i_1,\dots,i_f)\mapsto \left(1+\sum_{j=1}^f i_jp^{j-1}\right).$$
If $i$ corresponds to $(i_1,\dots,i_f)$, then $pi$ corresponds to
$(p-1-i_f,i_1,\dots,i_{f-1})$.  

The first letter of the pattern of $i$ is $u$ if and only if the first
element of the sequence $i_f,i_{f-1},\dots$ which is not equal to
$(p-1)/2$ is in fact $<(p-1)/2$.  More generally, if we have a word
$w=u^{e_1}l^{e_2}\cdots u^{e_k}$ where $k$ is odd, each $e_j>0$, and
$\sum e_j=f$,
then $i\in S$ has pattern $w$ if and only the following inequalities
are satisfied:
\begin{align*}
i_f\le(p-1)/2,&\quad i_{f-1}\le(p-1)/2,\ \dots\  \\
&\hfill i_{f-e_1+2}\le(p-1)/2,\quad i_{f-e_1+1}<(p-1)/2\\ \\
i_{f-e_1}\ge(p-1)/2,&\quad i_{f-e_1-1}\ge(p-1)/2,\ \dots\  \\
&\hfill i_{f-e_1-e_2+2}\ge(p-1)/2,\quad i_{f-e_1-e_2+1}>(p-1)/2\\
\vdots\\
i_{f-e_1-\cdots-e_{k-1}}\le(p-1)/2,&\quad
i_{f-e_1-\cdots-e_{k-1}-1}\le(p-1)/2,\ \dots\  \\
&\hfill i_{f-e_1-\cdots-e_k+2}\le(p-1)/2,\quad i_{f-e_1-\cdots-e_k+1}<(p-1)/2
\end{align*}

This leads to the following counts.

\begin{lemma}\label{lemma:counting}
\mbox{}\begin{enumerate}
\item Suppose $k>0$ is odd and $e_1,\dots,e_k$ are positive integers
  with $\sum e_j=f$.  Then the number of elements $i\in S$ whose
  pattern is $w=u^{e_1}l^{e_2}\cdots u^{e_k}$ is 
$$\pfrac{p-1}2^k\pfrac{p+1}2^{f-k}.$$
\item The number of $i\in S$ whose pattern starts $lu\cdots$ is
$$\pfrac{p-1}2\pfrac{p^{f-1}+1}2$$
and the number of $i\in S$ whose pattern starts $ll\cdots$ is 
$$\pfrac{p+1}2\pfrac{p^{f-1}-1}2$$
\end{enumerate}
\end{lemma}

\begin{proof}
  Part (1) follows immediately from the inequalities just before the
  lemma.  Part (2) is similar:  The pattern of $i$ starts $lu\dots$ if
  and only if $i_f>(p-1)/2$ and 
\begin{align*}
i_{f-1}&<(p-1)/2\\
\text{or}\quad i_{f-1}&=(p-1)/2\quad\text{and}\quad i_{f-2}<(p-1)/2\\
\text{or}\quad i_{f-1}&=i_{f-2}=(p-1)/2\quad\text{and}\quad i_{f-3}<(p-1)/2\\
&\vdots
\end{align*}
The number of such $i$ is 
$$\pfrac{p-1}2\left(\pfrac{p-1}2p^{f-2}+\cdots+\pfrac{p-1}2+1\right)
=\pfrac{p-1}2\pfrac{p^{f-1}+1}2.$$
Since the number of $i$ whose pattern starts with $l$ is clearly
$(p^f-1)/2$, the result for $ll$ follows by subtracting.
\end{proof}

\begin{proof}[End of the proof of Proposition~\ref{prop:H1-via-crys}]
  We saw above that the $\Fq$-dimension of $H^o_1$ is the number $i\in
  o$ whose word has the form $u\cdots l$, i.e., begins with $u$ and
  ends with $l$.  If the word associated to the standard base point in
  $o$ is $u^{e_1}l^{e_2}\cdots l^{e_{2k}}$ with $e_{i+k}=e_i$, then
  there are exactly $k$ elements $i\in o$ whose word has the form
  $u\cdots l$; if $i$ is the standard base point, they are 
$$i,p^{e_1+e_2}i,\dots,p^{e_1+\cdots+e_{2k-2}}i.$$

To compute the $\Fq$-dimension of $H_1$, we need only note that the
number of $i\in S$ whose word has the form $u\cdots l$ is the same as
the number of $i$ whose pattern starts $lu\cdots$.  Thus part
(2) of Lemma~\ref{lemma:counting} finishes the proof.
\end{proof}

\subsection{Proof of Theorems~\ref{thm:disc} and \ref{thm:index}} 
We now give the proofs of our results on the $o$-part of the
Mordell-Weil group $E(K_d)$.  We proved in \cite{ConceicaoHallUlmer14}
that $(E(K_d)\tensor\Zp)^o=0$ unless $o$ is an orbit with
$\gcd(o,d)<d/2$ and $p$ is balanced modulo $d/\gcd(d,o)$, so we make
those hypotheses for the rest of the subsection.

The first step is to note that Theorem~\ref{thm:reduction-to-S}(1) and
Theorem~\ref{thm:cohom-of-product}(1) imply that
$$\left(E(K_d)\tensor\Zp\right)^o\cong \left(H^1(\CC)\tensor_W
  H^1(\DD)\right)^{\Delta,o,F=p}.$$
This last group is denoted $H^o$ in Section~\ref{s:exercises}, where
we proved an isomorphism $H^o\cong\Gamma_o$.  

In order to prove the theorems, we need to consider $H^o$ as a
submodule of 
$$M^o:=\left(H^1(\CC)\tensor_W H^1(\DD)\right)^{\Delta,o}.$$
This is a free $W$-module on which the cup product induces a
perfect pairing.  The restriction of that pairing to $H^o$ corresponds
to the height pairing on $E(K_d)$, so to compute the discriminant of
the latter, it suffices to know the index of the $W$-submodule of
$M^o$ generated by $H^o$.  More precisely, the discriminant is $p^{2a}$
where 
$$a=\len_W\left(M^o/W H^o\right).$$
 
We saw above that $f_i,f_{pi},\dots,f_{p^{|o|-1}i}$ is a $W$-basis of
$M^o$.  Let $\eta_1,\eta_2,\dots,\eta_{|o|}$ be a $\Zp$-basis of $W$.
Then the classes
$$c_\ell=\sum_{j=0}^{|o|-1}p^{a_j}\sigma^j(\eta_\ell)f_{ip^j}\qquad\ell=1,\dots,|o|$$ 
form a $\Zp$-basis of $H^o$.  Here $j\mapsto a_j$ is the function
associated to $o$ in Subsection~\ref{ss:invs-of-orbits}.

In matrix form we have
$$\begin{pmatrix}c_1\\ \vdots \\ c_{|o|}\end{pmatrix}=
\begin{pmatrix}\sigma^0(\eta_1)&\sigma^1(\eta_1)&\dots&\sigma^{|o|-1}(\eta_1)\\
\sigma^0(\eta_2)&\sigma^1(\eta_2)&\dots&\sigma^{|o|-1}(\eta_2)\\
\vdots&\vdots&\ddots&\vdots\\
\sigma^0(\eta_{|o|})&\sigma^1(\eta_{|o|})&\dots&\sigma^{|o|-1}(\eta_{|o|})\end{pmatrix}
\begin{pmatrix}
p^{a_1}&0&\dots&0\\
0&p^{a_2}&\dots&0\\
\vdots&\vdots&\ddots&\vdots\\
0&0&\dots&p^{a_{|o|}}\end{pmatrix}
\begin{pmatrix}
f_i\\f_{pi}\\ \vdots\\ f_{p^{|o|-1}i}
\end{pmatrix}.$$
Since $W$ is unramified over $\Zp$, the determinant of the first
matrix on the right is a unit.  The determinant of the second matrix on
the right is clearly $p^{a_1+\cdots+a_{|o|}}$ and this is the length
of the quotient of $M^o$ by the $W$-span of $H^o$.  This proves that 
$$\disc\left(E(K_d)\tensor\Zp\right)^o=
p^{2(a_1+\cdots+a_{|o|})}$$
and this is the assertion of Theorem~\ref{thm:disc}.

To prove Theorem~\ref{thm:index}, note that we have containments 
$$V^o_d\subset E(K_d)^o\cong H^o\subset M^o$$
and we can compute the lengths of $M^o/WH^o$ and $M^o/WV_d^o$ via
discriminants.  

We just saw that
$$\len_W\frac{M^o}{WH^o}=a_1+\cdots+a_{|o|}.$$  
Let us simplify the sum using that we are in the complementary case,
so that $k$ is odd and $e_{k+j}=e_j$.  We have
\begin{align*}
\sum_{j=1}^{|o|}a_j&=\sum_{j=1}^{2k}(-1)^{j+1}\binom{e_j+1}2+e_je_{1,j-1}\\
&=\sum_{j=1}^{k}(-1)^{j+1}\binom{e_j+1}2+e_je_{1,j-1}
+\sum_{j=1}^{k}(-1)^{k+j+1}\binom{e_j+1}2+e_je_{1,k+j-1}\\
&=\sum_{j=1}^{k}e_j(e_{1,j-1}+e_{1,k+j-1})
\end{align*}
where the second equality uses that $e_{k+j}=e_j$ and the last
equality uses that $k$ is odd.  Noting that
$e_{1,j-1}+e_{1,k+j-1}=e_{1,k}=ht(o)$, we find that 
$$\sum_{j=1}^{|o|}a_j=\frac{|o|}2ht(o).$$

On the other hand, it follows from \cite[Thm~8.2]{Ulmer14a} (when
$d=p^f+1$) and \cite[Prop.~7.1]{ConceicaoHallUlmer14} (when
$d=2(p^f-1)$) that 
$$\len_W\frac{M^o}{WV_d^o}=\frac{|o|f}2.$$
Thus we have
$$\log_p[E(K_d)^o:V_d^o]=\len_W\frac{WH^o}{WV_d^o}=\frac{|o|}2\left(f-ht(o)\right).$$
Since $V_d^o\cong\Gamma_o$ and $\Gamma_o$ has a unique $G$-invariant
superlattice of index $p^{|o|e}$, namely $p^{-e}\Gamma_o$, we must
have
$$\frac{E(K_d)^o}{V_d^o}\cong
p^{-(f-ht(o))/2}\Gamma_o/\Gamma_o
\cong\Gamma_o/p^{(f-ht(o))/2}\Gamma_o.$$
Note also that when $\gcd(o,d)=1$, we have $f=\sum_{j=1}^ke_j$ and
$ht(o)=e_1-e_2+\cdots+e_k$, so
$(f-ht(o))/2=\sum_{j=1}^{(k-1)/2}e_{2j}$.  These are exactly the
assertions of Theorem~\ref{thm:index}, so this completes the proof.

\subsection{Proof of Theorem~\ref{thm:sha}} 
Let $\Fq$ be an extension of $\Fp(\mu_d)$ and consider $E$ over
$\Fq(u)$ with $u^d=t$.

The first step in the proof is to note that
Theorem~\ref{thm:reduction-to-S}(2) and
Theorem~\ref{thm:cohom-of-product}(2) give an isomorphism of
$\Zp[G]$-modules between $\sha(E/\Fq(u))[p^n]^o$ and the cokernel of
the map
$$\left(\left(H^1(\CC)\tensor_w
    H^1(\DD)\right)^{\Delta,o,F=p}\right)/p^n
\to \left(H^1(\CC)/p^n\tensor_W H^1(\DD)/p^n\right)^{\Delta,o,F=V=p}.$$
In the notation of Section~\ref{s:exercises}, this is the cokernel of
$$H^o/p^n\to H^o_n$$
and in Proposition~\ref{prop:H^o_n}(2), we showed that for all $n\ge
ht(o)$ this cokernel is
$$\frac{\bigoplus_{j=1}^kW_{d_j}(\Fq)}{W_{d_k}(\F_{p^{|o|}})}$$
where the $d_j$ are the invariants associated to $o$ in
Subsection~\ref{ss:more-invs}.  This is precisely part (1) of the
theorem.  Part (2) follows immediately once we note that if
$\gcd(o,d)=1$, then $\F_{p^{|o|}}=\F_{p^{2f}}=\Fp(\mu_d)$.

\subsection{Exponents}
We prove parts (1) and (2) of Theorem~\ref{thm:main}.  Clearly part
(2) implies part (1).

By Theorem~\ref{thm:index}, the exponent of $(E(K_d)/V_d)^o$ is
$p^{(f-ht(o))/2}$.  This is maximized when $ht(o)$ is minimized.  If
$f$ is odd, there is an $i\in\Z/d\Z$ with pattern $(ul)^f$ and the
corresponding word has height 1.  If $f$ is even, the minimum value of
$ht(o)$ is 2, which is achieved by an orbit with pattern (and word)
$(ul)^{f-1}uu(lu)^{f-1}ll$.  By Lemma~\ref{lemma:counting}, any such
word actual does arise as the word of some $i\in S$.  Thus the
exponent of $E(K_d)/V_d$ is $p^{\lfloor(f-1)/2\rfloor}$.

By Theorem~\ref{thm:sha}, the exponent of $\sha(E/K_d)^o$ is
$p^{d_{k-1}}$.  By Lemma~\ref{lemma:inv-factors3},
$$d_{k-1}=\max\{e_{ij}|2\le i\le j\le k-1,\text{$i$ and $j$ even}\}.$$
Clearly the alternating sum $e_i-e_{i+1}+\cdots$ is maximized when it
is a single term, and $d_{k-1}$ is maximized by a word whose first
half has the form $u^{e_1}l^{e_2}u^{e_3}$.  In order for this to be
the word associated to a good base point, we must have $e_1\ge e_2$
and $e_2\le e_3$.  Again, by Lemma~\ref{lemma:counting}, any such word
actual does arise as the word of some $i\in S$.  Thus for a given $f$,
the maximum value of $d_{k-1}=e_2$ is $\lfloor f/3\rfloor$ and the
exponent of $\sha(E/K_d)$ is $p^{\lfloor f/3\rfloor}$

\subsection{Comparison of $E/V$ and $\sha$}
Now we prove parts (3) and (4) of Theorem~\ref{thm:main}.

For part (3), note that when $f=1$ or $2$, up to rotation all words
have the form $u^fl^f$ and by Theorems~\ref{thm:index} and
\ref{thm:sha} the groups under discussion are trivial in these cases.
If $f=3$, up to rotation every word is $u^3l^3$ or $(ul)^3$.  In the
latter case, both $\left((E(K_d)/V_d)^o\right)^2$ and $\sha(E/K_d)^o$
are isomorphic to $(\Gamma_o/p)^2$.  When $f=4$, up to rotation the
possible words are $u^4l^4$ and $u^2lul^2ul$.  In the former case both
$\left((E(K_d)/V_d)^o\right)^2$ and $\sha(E/K_d)^o$ are trivial, and in the latter
they are both isomorphic to $(\Gamma_o/p)^2$.

For part (4), we note that by Proposition~\ref{prop:G-reps}
$\Gamma_o/p$ is an absolutely irreducible $\Zp[G]$-module.  Thus all
Jordan-H\"older factors of $(E(K_d)/V_d)^o$ and $\sha(E/K_d)^o$ are
$\Gamma_o/p$, and to prove part (4) it suffices to count the
multiplicities.  By Theorem~\ref{thm:index}, the multiplicity for
$(E(K_d)/V_d)^o$ is $(f-ht(o))/2$.  By Theorem~\ref{thm:sha}, that for
$\sha(E/K_d)^o$ is $d_1+\cdots+d_{k-1}$.  But from the definition,
$$\sum_{j=1}^kd_j=\sum_{j=1}^k e_{2j-1}=\sum_{j=1}^ke_j=f.$$
(Here we use that we are in the complementary case, so $k$ is odd and
$e_{j+k}=e_j$.)  As noted just after Lemma~\ref{lemma:inv-factors2},
$d_k=ht(o)$, so the total multiplicity of $\Gamma_o/p$ in
$\sha(E/K_d)^o$ is $f-ht(o)$.  This completes the proof of part (4).

\subsection{Polynomial interpolation of orders}\label{ss:interpolation}
Now we prove part (5) of Theorem~\ref{thm:main}.
Write $\inv(o)$ for $|o|(f-ht(o))$ so that
$|\sha(E/K_d)^o|=p^{\inv(o)}$.  Then $|\sha(E/K_d)|=p^I$ where
$$I=\sum_{o\in O_{d,p}}\inv(o).$$

Recall that a word is ``good'' if it associated to a good base point
of an orbit.  Let $|\aut(w)|$ be the number of automorphisms of $w$,
i.e., the number of rotations leaving $w$ invariant.  Then since
$\inv(o)$ only depends on the word associated to $o$, we have
$$I=\sum_{\text{good }w}\frac{\left|\left\{i\left|\,\text{the orbit
      through $i$ is $w$}\right.\right\}\right|}{|\aut(w)|}\inv(w).$$
Now $\inv(w)/|\aut(w)|$ is the same for a word $w$ as for the
concatenation of several copies of $w$, so we may take the sum only
over full length words and consider $i$'s whose \emph{pattern} is $w$,
where pattern is defined as in Subsection~\ref{ss:counting}.  Then 
$$I=\sum_{\text{full length, good
  }w}\frac{\inv(w)}{|\aut(w)|}\left|\left\{i\left|\,\text{the pattern
        of $i$ is $w$}\right.\right\}\right|.$$

To finish, we note that by Lemma~\ref{lemma:counting},
$\left|\left\{i\left|\,\text{the pattern 
        of $i$ is $w$}\right.\right\}\right|$ is a polynomial in $p$.
This shows that there is a polynomial $F_f$ depending only on $f$ with
coefficients in $\Z[1/2]$ such that $I=F_f(p)$.  It also shows that
when $I$ is not zero, (i.e., when there are words with non-zero
invariant, i.e., when $f\ge3$), the degree of $F_f$ is $f$.

Here is an example:  If $f=3$, the good words are $u^3l^3$, $ululul$, and
$ul$. We have $\inv(u^3l^3)=0$, $\inv(ululul)=12$ and $\inv(ul)=4$.
Using Lemma~\ref{lemma:counting}, we find that 
$$I=\frac{12}3\left(\frac{p-1}2\right)^3=\frac{(p-1)^3}2.$$
It looks like an interesting and perhaps difficult problem to give a
closed expression for $F_f$ in general.

\section{Complements}\label{s:complements}
In the last section of the paper, we give four complementary results.
Two of them recover much of the main theorem (specifically, the
$p$-torsion in $\sha(E/K_d)$ and $(E(K_d)/V_d)$) using flat rather
than crystalline cohomology.  This gives a reassuring check on the
combinatorial aspects of the main results.  The third gives an
extension of many of the results of the paper to characteristic $p=2$.
In the fourth, we briefly touch upon a generalization to higher genus
curves.

\subsection{$p$-torsion in $\sha(E/K_d)$ via flat
  cohomology}\label{ss:Selp} 
It is possible to compute the $p$-Selmer group of $E/K_d$ (and
therefore the $p$-torsion in the Tate-Shafarevich group) using flat
cohomology and the methods of \cite{Ulmer91}.  This yields a second
proof that $\sha(E/K_d)$ is trivial if and only if $f\le2$, and it
provides a check on the crystalline calculation described in the main
part of the paper.

We refer to \cite[\S1]{Ulmer91} for the definition of the Selmer group
denoted $\Sel(K_d,p_E)$.  It sits in an exact sequence
$$0\to E(K_d)/pE(K_d)\to\Sel(K_d,p_E)\to\sha(E/K_d)[p]\to0.$$

\begin{prop}\label{prop:sel-p-dim}
With $p$, $f$, $d=p^f+1$, and $E$ as in the rest of the
  paper,
\begin{enumerate}
\item $\Sel(K_d,p_E)$ is an $\Fp$-vector space of dimension
  $(p-1)(p^{f-1}+1)f/2$.
\item $\sha(E/K_d)=0$ if and only if $f\le2$.
\end{enumerate}
\end{prop}

The proof of the proposition will occupy the rest of this section.
Note that part (2) follows easily from part (1) since we know that
  $E(K_d)/pE(K_d)$ is an $\Fp$-vector space of dimension $p^f-1$.

Let $A=A(E,dx/2y)$ be the Hasse invariant of $E$.  By a simple
calculation (see, e.g., \cite[\S13, Prop.~3.5]{HusemollerEC}), this is
$$A=\sum_{i=0}^{(p-1)/2}\binom{(p-1)/2}{i}^2t^i.$$
Let $\alpha$ be a $(p-1)$-st root of $A$ in
$\overline{K}$ and let $F_{d,p}$ be the field $K_d(\alpha)$.  Then
$F_{d,p}$ is a Galois extension of $K_d$ with group $\Fptimes$.  We
let $I_{d,p}\to\P^1_u$ be the corresponding cover of smooth projective
curves over $\Fq=\Fp(\mu_d)$.  (Here $I$ is for ``Igusa.'')  Then
the argument leading to \cite[Thm.~7.12b]{Ulmer91} yields an
isomorphism 
$$\Sel(K_d,p_E)\cong H^0(I_{d,p},\Omega^1_{I_{d,p}})^{\psi^{-1},\,\CC=0}$$
where $\CC=0$ indicates the kernel of the Cartier operator (i.e., 
the subspace of exact differentials), and 
$\psi^{-1}$ denotes the subspace where
$\gal(F_{d,p}/K_d)=\Fptimes$ acts via the character $\psi^{-1}$ where
$\psi:\Fptimes\to k^\times$ is the natural inclusion.

(Some of the results of \cite{Ulmer91} used just above are
stated for $p>3$, but this is assumed only to guarantee that at
places of potentially multiplicative reduction, $E$ obtains
multiplicative reduction over an extension of degree prime to $p$.
This is true for the Legendre curve even when $p=3$.)

Using the covering $I_{d,p}\to\P^1_u$ (which is ramified exactly where
$\alpha$ has zeroes), we find that 
$$H^0(I_{d,p},\Omega^1_{I_{d,p}})^{\psi^{-1}}=
\left\{\left.\frac{f(u)\,du}{\alpha^{p-2}}\,\right|\,\deg(f)\le
  N\right\}$$ where $f$ is a polynomial of degree at most
$N=(p-2)(p^f+1)/2-2$ when $d=p^f+1$.  (For $d=1$, there is also
ramification at infinity and we have $N=(p-5)/2$.)  The crux of the
proof is to compute the subspace killed by the Cartier
operator.

To that end, we first make some calculations at level $d=1$, i.e., on
the curve $I_{1,p}$.  Write
$$\frac{dt}{\alpha^{p-2}}=\left(f_0^p+tf_1^p+\cdots+t^{p-1}f_{p-1}^p\right)dt$$
where the $f_i\in F_{1,p}=\Fp(t,\alpha)$.  Since
$(1/\alpha)=(A/\alpha^p)$, the $f_i$ are all polynomials in $t$ times
$1/\alpha^{p-2}$.  Note that $\CC(t^idt/\alpha^{p-2})=f_{p-1-i}dt$ for
$i=0,\dots,p-1$.  

The key step in the proof of the proposition is the following
calculation of dimensions of certain spaces spanned by the $f_i$.  In
it, we use angle brackets to denote the $\Fq$-span of the terms
within.

\begin{lemma}\label{lemma:relations}\mbox{}\begin{enumerate}
\item $\dim_\Fq\left\<f_{p-1},f_{p-2},\dots,f_{(p+3)/2}\right\>=(p-3)/2$.
\item We have equalities and containments
\begin{align*}
\left\<f_{p-1},\dots,f_{(p+3)/2}\right\>&=\left\<f_{p-2},\dots,f_{(p+1)/2}\right\>
=\cdots=\left\<f_{(p-1)/2},\dots,f_{2}\right\>\\
&\subsetneq\left\<f_{(p-1)/2},\dots,f_{1}\right\>=
\left\<f_{(p-3)/2},\dots,f_{0}\right\>
\end{align*}
and
\begin{align*}
  \left\<f_{(p-3)/2},\dots,f_{0}\right\>&=\left\<f_{(p-5)/2},\dots,f_{0},tf_{p-1}\right\>
=\cdots=\left\<f_{0},tf_{p-1},\dots,tf_{(p+3)/2}\right\>
\end{align*}
\end{enumerate}
\end{lemma}

\begin{proof}
  Recall that $K=K_1=\Fp(t)$.  First, we note that $E(K)/pE(K)=0$ by
  \cite[5.2 and 6.1]{Ulmer14a}, and using the BSD formula as in
  \cite[\S10]{Ulmer14a} shows that $\sha(E/K)=0$.  Thus
  $\Sel(K,p_E)=0$.

On the other hand, as we noted above, $\Sel(K,p_E)$ is isomorphic to
the kernel of the Cartier operator on 
$$\left\{\left.\frac{f(t)\,dt}{\alpha^{p-2}}\,\right|\,\deg(f)\le
    (p-5)/2\right\}.$$ Since this kernel is trivial, we find that
  $f_{p-1},\dots,f_{(p+3)/2}$ are linearly independent, and this is
  the first claim of the lemma.

Now set $g_0=-A'=-dA/dt$ and $g_i=iA-tA'$, and compute that
$A'dt=-\alpha^{p-2}d\alpha$ so that $d\alpha=g_0dt/\alpha^{p-2}$ and
$d(t^i\alpha)=t^{i-1}g_idt/\alpha^{p-2}$ for $i\ge0$.  These exact
differentials provide relations among the $f_i$.  More precisely,
note that $g_0$ has degree $(p-3)/2$ and non-zero constant term, so
$\CC(g_0dt/\alpha^{p-2})=0$ implies that a linear combination of
$f_{p-1},\dots,f_{(p+1)/2}$ is zero, and $f_{p-1}$ and $f_{(p+1)/2}$
appear in this relation with non-zero coefficients.  This implies that 
 $$\left\<f_{p-1},\dots,f_{(p+3)/2}\right\>=\left\<f_{p-2},\dots,f_{(p+1)/2}\right\>$$
which is the first equality displayed in part (2) of the lemma.  

To obtain the rest of the equalities in that display, we set $h_0=g_0$
and 
\begin{align*}
h_i&=\binom{(p-1)/2}i^2t^{i-1}g_i+h_{i-1}\\
&=\sum_{\ell=1}^i\binom{(p-1)/2}{\ell}^2t^{\ell-1}g_\ell+g_0
\end{align*}
for $i=1,\dots,(p-3)/2$.  One checks inductively that $h_i$ has degree
$(p-3)/2+i$ and its non-zero term of lowest degree is
$-(i+1)\binom{(p-1)/2}{i+1}^2t^i$.  Thus $\CC(h_idt/\alpha^{p-2})=0$
gives a relation among $f_{p-1-i},\dots,f_{(p+1)/2-i}$ where the
coefficients of $f_{p-1-i}$ and $f_{(p+1)/2-i}$ are non-zero.  These
relations give the desired equalities between spans. 

The proper containment in the second line of the first display in
part~(2) of the lemma is equivalent to saying that $f_1$ and
$\left\<f_{(p-1)/2},\dots,f_{2}\right\>$ are linearly
independent.  One way to see this is to note that the
$\alpha^{p-2}f_i$ are polynomials in $t$ and since the degree of
$A^{p-2}$ is congruent to 1 modulo $p$, $\alpha^{p-2}f_1$ has degree
strictly greater than $\alpha^{p-2}f_i$ for $i=2,\dots,p-1$.  Thus $f_1$
and $\<f_{p-1},\dots,f_{2}\>$ are linearly independent.

To obtain the remaining equalities of part~(2), we consider
the exact differentials $t^{i-1}g_idt/\alpha^{p-2}$ for
$i=(p+1)/2,\dots,p-1$.  In this range, $t^{i-1}g_i$ has degree
$(p-3)/2+i$ and lowest term of degree $i-1$.  For $i=(p+1)/2$, we get a
relation among $f_{(p-1)/2},\dots,f_0$ with $f_{(p-1)/2}$ and $f_0$
appearing, yielding the last equality in the first display of
part~(2).  For $i=(p+3)/2,\dots,p-1$, we get relations among
$f_{p-i},\dots,tf_{(3p+1)/2-i}$ with $f_{p-i}$ and $tf_{(3p+1)/2-i}$
appearing, and these relations give the equalities in the second
display of part~(2).

This completes the proof of the lemma.
\end{proof}

We are now in position to compute the rank of the
Cartier operator on $H^0(I_{d,p},\Omega^1_{I_{d,p}})^{\psi^{-1}}$, in
  other words
$$R:=\dim_\Fq\CC\left(\left\{\left.\frac{f(u)\,du}{\alpha^{p-2}}
\,\right|\,\deg(f)\le(p-2)(p^f+1)/2-2\right\}\right).$$
Noting that $u=t/u^{p^f}$ and $du=u^{-p^f}dt$, we find that
$$\CC(u^{i+pj}du/\alpha^{p-2})=u^{j-(i+1)p^{f-1}}f_i\,du$$
for $0\le i\le p-1$ and
$$0\le j\le \begin{cases}
\frac{p-3}2p^{f-1}+\frac{p^{f-1}-1}2&\text{if $i\le p-3$}\\
\frac{p-3}2p^{f-1}+\frac{p^{f-1}-3}2&\text{if $i=p-2,p-1$.}
\end{cases}
$$
This implies that the image of $\CC$ will be spanned by spaces of the
form $u^e\<f_a,\dots,f_b\>$.  To compute the dimension, we observe
that if $e_1,\dots,e_\ell$ are integers pairwise non-congruent modulo
$d$, and if $V_1,\dots V_\ell$ are $\Fq$-vector spaces spanned by
subsets of $\{t^jf_i\,|\,0\le i\le p-1,\ j\in\Z\}$, then the subspaces
$u^{e_i}V_i$ of $F_{d,p}$ are linearly independent over $\Fq$.
This plus the information in Lemma~\ref{lemma:relations} suffices to
compute $R$.

An elaborate and somewhat unpleasant exercise in bookkeeping which we
omit leads to 
$$R=\frac{(p-3)}2\frac{(p-1)}2p^{f-1}+\frac{(p-3)}2\frac{(p^f+3)}2
+\frac{p-1}2\left(p^{f-1}-1\right)$$
which in turn implies that 
$$\dim_\Fq\ker(\CC)=N+1-R=\frac{(p-1)}2\frac{(p^{f-1}+1)}2.$$
Since $[\Fq:\Fp]=2f$, this completes the proof of
Proposition~\ref{prop:sel-p-dim}.  

  The analysis above yields quite a bit more information about
  $\Sel(K,p_E)$:
\begin{corss}\label{cor:diffs}
The differentials
$$\omega_{i,j}=u^{pj-ip^f}h_i(t)du/\alpha^{p-2}
=u^{i+pj}t^{-i}h_i(t)du/\alpha^{p-2}$$ 
for $0\le i\le (p-3)/2$ and $0\le j\le(p^{f-1}-1)/2$ are regular and
exact, and they give an $\Fq$ basis for 
$$\Sel(K,p_E)\cong H^0(I_{p,d},\Omega^1_{I_{p,d}})^{\psi^{-1},\CC=0}.$$
\end{corss}

\begin{proof}
  The proof of Proposition shows that the displayed differentials are
  exact and lie in the $\psi^{-1}$ eigenspace.  They are obviously
  linearly independent, and since the number of them is the dimension
  of $\Sel(K,p_E)$ over $\Fq$, they form an $\Fq$ basis.
\end{proof}

We can also deduce results on the structure of $\Sel(K,p_E)$ as a
module over $\Fp[G]$:

\begin{corss}\label{cor:sha-p}
  If $o\in O$ is an orbit whose pattern is $u^{e_1}l^{e_2}\cdots
  u^{e_k}$ then the multiplicity of $\Gamma_o/p$ in $\Sel(K,p_E)$ is
  $k$, and its multiplicity in $\sha(E/K_d)$ is $k-1$.
\end{corss}

\begin{proof}
  The previous corollary shows that as an $\Fp[G]$-module,
  $\Sel(K_d,p_E)$ is the direct sum 
$$\bigoplus_{\substack{0\le i\le(p-3)/2\\0\le j\le(p^{f-1}-1)/2}}
\Fq u^{i+pj}.$$
If $\ell\in o$, then by Proposition~\ref{prop:G-reps}(5),  $\Fq
u^\ell\cong(\Gamma_o/p)^{2f/|o|}$.

Now an orbit $o$ appears in the discussion above as many times as
there are $\ell\in o$ which can be written $\ell=1+i+pj$ with $0\le
i\le(p-3)/2$ and $0\le j\le(p^{f-1}-1)/2$.  
Writing $\ell=\sum_{k=1}^f i_kp^{k-1}$ as in
Subsection~\ref{ss:counting}, we see that $\ell$ can be written
$\ell=1+i+pj$ with $i$ and $j$ ``small'' in the sense above if and
only if the word associated to $\ell$ begins and ends with the
letter $u$.  Thus, if the word of $o$ is
$u^{e_1}\cdots u^{e_{k'}}l^{e_1}\cdots l^{e_{k'}}$, then the number of
times $o$ arises is $k'$. 

To finish, we note that the pattern of the standard base point of
$o$ is the first half of $w(o)^{2f/|o|}$, and written in exponential
form this has $k=k'(2f/|o|)$ runs of $u$s.  Thus $\Gamma_o/p$ appears
$k$ times in $\Sel(K_d,p_E)$.  This proves our claim about
$\Sel(K_d,p_E)$. 

The claim about $\sha(E/K_d)$ follows from the fact
that as an $\Fp[G]$-module, $E(K_d)/p$ is the direct sum of all
$\Gamma_o/p$ with $o\in O$, each taken with multiplicity one.  (This
follows immediately from Remark~\ref{rem:E(K_d)-as-G-module}.)
\end{proof}

We need one more result coming from \cite{Ulmer91}. To state it,
recall that the Selmer group for the isogeny $\Fr:E\to E^{(p)}$ over
$F_{p,d}$ is naturally a subgroup of 
$$F_{p,d}^\times/F_{p,d}^{\times p}\cong\Omega^1_{log}(F_{p,d}),$$ 
where the latter is the
space of meromorphic, logarithmic differentials on $I_{p,d}$.  In
\cite[\S5]{Ulmer91}, we defined a logarithmic differential $dq/q$
attached to $E/F_{p,d}$ which depends only on the choice of a $p-1$-st
root $\alpha$ of $A$ (or what amounts to the same thing, a non-trivial point
of order $p$ in $E^{(p)}(F_{p,d})$).

\begin{lemma}\label{lemma:dq/q}
We have an equality
$$\frac{dq}q=\frac{\alpha^2du}{u(t-1)}=\frac{\alpha^2du}{u(u^d-1)}$$
 of meromorphic differentials on $I_{p,d}$ and a calculation of
Selmer groups: 
$$\Sel(F_{p,d},\Fr_E)=\Fp\frac{dq}q.$$
\end{lemma}

\begin{proof}
  The same argument as in \cite[Thm.~7.6]{Ulmer91} shows that the
  Selmer group $\Sel(F_{p,d},\Fr_E)$ is isomorphic to the group of
  logarithmic differentials with simple poles at places where $E$ has
  multiplicative reduction and zeros of order $p$ at places where $E$
  has supersingular reduction.  An easy exercise using the covering
  $I_{p,d}\to\P^1_u$ shows that the only such differentials are the
  $\Fp$-multiples of $\alpha^2du/u(t-1)$.  Since $dq/q$ lies in this
  Selmer group (as the image of the chosen point of order $p$ on
  $E^{(p)}(F_{p,d})$), it is a non-zero multiple of
  $\alpha^2du/u(t-1)$.  Which multiple it is will not be material for
  what follows, so we omit the check that $dq/q$ is
  $\alpha^2du/u(t-1)$ on the nose.
\end{proof}

\subsection{$p$-torsion in $E(K_d)/V_d$ via flat
  cohomology}\label{ss:coboundary} 
The results of \cite{Ulmer91} and \cite{Broumas97} also afford good
control on the $p$-torsion in $E(K_d)/V_d$.  We continue with the
notation of the previous subsection.  In particular, we assume that
$d=p^f+1$. 

We state our result in terms of the decomposition of $E(K_d)/V_d$ as a
module over $\Zp[G]$ (in fact over $\Fp[G]$ since we are concerned
only with the $p$-torsion).

\begin{prop}\label{prop:(E/V)[p]}
We have
$$\ker\left(p:E(K_d)/V_d\to E(K_d)/V_d\right)^o=\begin{cases}
\Gamma_o/p&\text{if the word of $o$ is not $u^fl^f$}\\
0&\text{if the word of $o$ is $u^fl^f$}
\end{cases}$$
\end{prop}

\begin{proof}
First, we note that an easy application of the snake lemma shows that 
$$\ker\left(p:E(K_d)/V_d\to E(K_d)/V_d\right)\cong
\ker\left(V_d/p\to E(K_d)/p\right).$$ 
Moreover, we have an injection
\begin{equation}\label{eq:coboundary}
E(K_d)/p\into\Sel(K_d,p_E)
\end{equation}
and so it will suffice to compute the kernel of the composed map
$V_d/p\to\Sel(K_d,p_E)$.  We will do this by using Broumas's wonderful
formula for \eqref{eq:coboundary} and the explicit calculation of
$\Sel(K_d,p_E)$ in the preceding subsection.

Recall that $V_d/p$ is isomorphic as $\Fp[G]$-module to $\oplus_{o\in
  O}\Gamma_o/p$ and that this $\Fp[G]$-module is cyclic, generated by
the point $P(u)=(u,u(u+1)^{d/2})$ defined in \cite[\S3]{Ulmer14a}.

As noted in the previous section, we have 
$$\Sel(K_d,p_E)\cong
H^0(I_{p,d},\Omega^1_{i_{p,d}})^{\CC=0,\psi^{-1}}.$$
Using \cite[Prop.~5.3]{Ulmer91}, the space of exact differentials
above can be identified with a subgroup of the additive group of $K_d$
via the map $\omega\mapsto\alpha^p\omega/(dq/q)$ where $dq/q$ is the
differential computed in Lemma~\ref{lemma:dq/q} and $\alpha$ is a root
of $\alpha^{p-1}=A$.  The Main Theorem of \cite{Broumas97} gives an
explicit formula for the composition
$$\mu: E(K_d)\to\Sel(K_d,p_E)\to K_d.$$

To state the result, write
$$\left(x(x+1)(x+t)\right)^{(p-1)/2}=x^pM(x)+Ax^{p-1}+\text{ lower
  order terms}$$
and let $\wp_A(z)=z^p-Az$.  Then (after a considerable amount of
boiling down), Broumas's formula says
$$\mu(P(u))=u(u+1)^{(p^f+1)/2}M(u)-\wp_A\left(u(u+1)^{(p^f-1)/2}\right).$$
(We note that there is a typo in Broumas's paper in the case $p=3$.
Namely, in formula (36) on page 140 of \cite{Broumas97}, 
$2\mathcal{D}a_2/a_2+\mathcal{D}a_6/a_6$ should be replaced with 
$\left(2\mathcal{D}a_2/a_2+\mathcal{D}a_6/a_6\right)x$.)

The last displayed quantity is an element of the polynomial ring
$\Fq[u]$, and we are going to compute it modulo the ideal generated by
$t=u^d$.

To see that this will suffice for our purposes, recall from
Corollary~\ref{cor:diffs} the exact differentials $\omega_{i,j}$ giving
a $\Fq$ basis for the Selmer group.  Using Lemma~\ref{lemma:dq/q},
we find that
$$f_{i,j}:=\alpha^p\omega_{i,j}/(dq/q)=u^{1+i+pj}t^{-i}h_i(t)(t-1)$$
for $0\le i\le(p-3)/2$ and $0\le j\le(p^{f-1}-1)/2$.  Thus in order to
write $\mu(P(u))$ in terms of the $f_{i,j}$, it suffices to know
$\mu(P(u))$ modulo $t$.

Straightforward computation from the definition shows that 
$$M(u)\equiv\frac{(u+1)^{(p-1)/2}-1}u\quad\text{and}\quad 
A\equiv1\pmod{t\Fq[u]}.$$
Thus
\begin{align*}
\mu(P(u))&\equiv(u+1)^{(p^f+p)/2}-(u+1)^{(p^f+1)/2}
-u^p(u^p+1)^{(p^f-1)/2}+u(u+1)^{(p^f-1)/2}\\
&=(u+1)^{(p^f-p)/2}\left((u+1)^p-(u+1)^{(p+1)/2}-u^p(u^p+1)^{(p^f-p^{f-1})/2}
+u(u+1)^{(p-1)/2}\right)\\
&\equiv(u+1)^{(p^f-p)/2}\left(1-(u+1)^{(p-1)/2}\right)\\
&=-u\left(\sum_{j=0}^{(p-3)/2}\binom{(p-1)/2}{i+1}u^i\right)
(1+u^p)^{(p-1)/2}\cdots(1+u^{p^{f-1}})^{(p-1)/2}.
\end{align*}
(To pass from the second line to the third, note that the sum of the
first and third terms inside the large parentheses is congruent to 1
modulo $t$.)

The last expression makes it clear that $\mu(P(u))\pmod{t\Fq[u]}$ is
the sum of terms $cu^\ell$ where $u^\ell$ appears with non-zero
coefficient if and only if $\ell=1+\sum i_kp^{k-1}$ with
$i_1\le(p-3)/2$ and $i_k\le(p-1)/2$ for $2\le k\le f$.  It follows
that $\mu(P(u))$ is a linear combination (with non-vanishing
coefficients) of the $f_{i,j}$ where $\ell=1+i+pj$ satisfies the
same condition.

Now by Proposition~~\ref{prop:G-reps}(5), the $\Fp[G]$-modules $\Fq
u^\ell$ with $\ell$ satisfying the conditions just above are pairwise
non-isomorphic.  Thus the $\Fp[G]$-submodule of the Selmer group
generated by $\mu(P(u))$ is the direct sum of the corresponding
$\Gamma_o/p$.  The orbits in question are precisely those with word
$u^fl^f$, and this shows that the image of $V_d/p\to E(K_d)/p$ is
isomorphic to
$$\bigoplus_{\substack{o\in O\\w(o)=u^fl^f}}\Gamma_o/p.$$
The kernel is thus the sum of the $\Gamma_o/p$ where $o$ runs through
orbits with words not equal to $u^fl^f$.  This completes the proof of
the proposition.
\end{proof}

The proposition allows us to recover large parts of
Theorem~\ref{thm:main}:  It shows that $(E(K_d)/V_d$ is non-trivial if
and only if $f>2$, and together with Corollary~\ref{cor:sha-p} it
shows that $\sha(E/K_d)$ is not isomorphic to $(E(K_d)/V_d)^2$ as an
abelian group if $f>4$.

\subsection{An extension to $p=2$}
In this subsection we explain how the main results of the paper can be
extended to the case where $p=2$.

To that end, let $p$ be an arbitrary prime number and let $E'$ be the
elliptic curve over $K'=\Fp(t')$ defined by
$$y^2+xy+t'y=x^3+t'x^2.$$
As explained in \cite[\S11]{Ulmer14a} and
\cite[\S11]{ConceicaoHallUlmer14}, if $p>2$ and we identify $K'$ and
$K$ by sending $t'$ to $t/16$, then $E$ and $E'$ are 2-isogenous.
Moreover, for $d=p^f+1$, the fields $K'_d=\Fp(\mu_d,t^{\prime 1/d})$
and $K_d=\Fp(\mu_d,t^{1/d})$ can be identified as extensions of $K$.
Having done so, one finds that the subgroup $V'_d\subset E'(K'_d)$
defined in \cite[8.10(3)]{Ulmer13a} is carried over to $V_d\subset
E(K_d)$.  It follows that Theorem~\ref{thm:main} and its refinements
in Section~\ref{s:refined-results} hold for $E'(K'_d)/V'_d$ and
$\sha(E'/K'_d)$.

Now the equation above also defines an elliptic curve when $p=2$.
Moreover, the N\'eron model of $E'/K'_d$ is dominated by a product of
curves (two copies of the curve $\CC'$ over $\Fp(\mu_d)$ defined by
$z^d=x(1-x)$), see \cite[11.2(5)]{ConceicaoHallUlmer14}.  Thus the
methods of this paper may be used to compute $E'(K'_d)/V'_d$ and
$\sha(E'/K'_d)$ as modules over $\Zp[\gal(K'_d/K)]$.  Most of the
results have the same form and the proofs are mostly parallel, so we
will briefly discuss some of the differences and then state the
results.

The analogue of the geometric analysis leading to
Theorem~\ref{thm:reduction-to-S}  gives an isomorphism 
$$\left(E'(K'_d)/tor\right)\tensor\Z[1/d]\isoto
\left(\NS'(\CC\times\CC)\tensor\Z[1/d]\right)^{\mu_d}$$ where the
$\mu_d$ in the exponent is acting anti-diagonally. (In fact the most
natural way to state this would be with the arrow going the other way
and with the target being the subgroup of $E'(K'_d)$ generated by the
point in \cite[Thm.~8.1(2)]{Ulmer13a} and its Galois conjugates.  This
subgroup is free of rank $d-1$ and is a complement to the torsion
subgroup.)  The analogue of the isomorphism of Tate-Shafarevich and
Brauer groups in Theorem~\ref{thm:reduction-to-S}(2) goes through for
$E'$ without change.

The analysis of the arithmetic of a product in
Section~\ref{s:arithmetic-of-CxD} was done there also for $p=2$, and
the description of the cohomology of $\CC$ in Section~\ref{s:H1(C)}
works for $\CC'$ as well with very minor changes.  The $p$-adic
exercises in Section~\ref{s:exercises} also works essentially unchanged.

Altogether, one finds that the obvious analogues of
Theorem~\ref{thm:main} parts (1) through (4) hold for $E'/K'_d$.
Similarly for the refined Theorems~\ref{thm:index} and \ref{thm:sha}. 

There are a few differences to report as well.  For example, part (5)
of Theorem~\ref{thm:main} does not extend to $p=2$.  Indeed, the
polynomial appearing there does not even take an integral values at
$p=2$.  The correct statement can be deduced from the proof in
Subsection~\ref{ss:interpolation} by noting that the number of
elements in $\Z/d\Z\setminus\{0\}$ with a given pattern is 1 (rather
than $(p-1)^a(p+1)^b/2^f$ as in Lemma~\ref{lemma:counting}).

The results of Subsections~\ref{ss:Selp} and \ref{ss:coboundary} also
extend to $E'$.  One finds that the order of $\Sel(K'_d,p_{E'})$ has
order $2^{f-1}f+1$.  
The refined results of Corollary~\ref{cor:sha-p} and
Proposition~\ref{prop:(E/V)[p]} hold as stated.  However, the details
of the $2$-descent have a different flavor because $E'$ has a
$2$-torsion point over $K'$ so the kernel of $p$ is the direct sum of
the kernels of Frobenius and Verschiebung and the differential $dq/q$
is zero.  We leave the details as an exercise for the
interested reader.

\subsection{Higher genus}
Let $p$ be a prime number, $r$ and $d$ integers relatively prime to
$p$, and consider the curve $X$ defined by
$$y^r=x^{r-1}(x+1)(x+t)$$
over $\Fp(t)$ and its extensions $\Fq(u)$ with $u^d=t$.  The genus of
$X$ is $r-1$, and its Jacobian $J$ has interesting arithmetic over
$\Fq(u)$ for many values of $d$.

For simplicity we will only discuss the case where $r$ divides $d$,
$d=p^f+1$, and $\Fq=\Fp(\mu_d)$.  We write $K_d$ for $\Fq(u)$.  In 
\cite{AIMgroup}, explicit divisors are given on $X$ whose classes
in $J(K_d)$ generate subgroup $V_d$ of rank $(r-1)(d-2)$ and finite,
$p$-power index.  Moreover, it is shown there that we have a class
number formula
$$|\sha(J/K_d)|=[J(K_d):V_d]^2.$$

Most of the results of this paper extend to this situation and give an
explicit calculation of $\sha(J/K_d)$ and $J(K_d)/V_d$ as modules over
the group ring $\Zp[G]$ where $G=\mu_d\sdp\gal(\Fq/\Fp)$.  

Indeed, we saw in Subsection~\ref{ss:Xr} that the minimal regular
model $\XX\to\P^1_u$ of $X/K_d$ is birational to the quotient of a
product of curves by a finite group.  The product is $\SS=\CC\times\CC$
where $\CC$ is the smooth proper curve over $\Fq$ defined by
$z^d=x^r-1$.  We deduce from this a connection between the
Mordell-Weil and Tate-Shafarevich groups of $J$ and the N\'eron-Severi
and Brauer groups of $\SS$, as at the end of Subsection~\ref{ss:Xr}.
These groups are described in crystalline terms in
Section~\ref{s:arithmetic-of-CxD}. 

As we saw in Subsection~\ref{ss:Cr}, the crystalline cohomology of
$\CC$ breaks up into lines indexed by the set
$$S=\left\{\left.(i,j)\in(\Z/d\Z)\times(\Z/r\Z)\right|
  i\neq0,j\neq0,\<i/d\>+\<j/r\>\neq1\right\}.$$
The subspace $H^0(\CC/\Zp,\Omega^1_{\CC/\Zp})$ is generated by the
lines indexed by $(i,j)$ with $\<i/d\>+\<j/r\><1$. Calling this subset
$A$ and letting $B=S\setminus A$, we may use $A$ and $B$ to define
words associated to orbits of $\<p\>$ acting diagonally on $S$ and to
define a notion of balanced, as discussed at the end of
Subsection~\ref{ss:Cr}.

The $p$-adic exercises of Section~\ref{s:exercises} go through
essentially unchanged, and interpreting ``balanced'' as above, we find
that Theorem~\ref{thm:main} parts (1) through (4), and the refined
results in Theorems~\ref{thm:disc}, \ref{thm:index}, and \ref{thm:sha}
hold as stated.  An interpolation result, as in part (5) of
Theorem~\ref{thm:main}, also holds with a polynomial $F$ which depends
on $r$ and $f$, but not on $p$.

Exploring the arithmetic of $J$ for other values of $r$ and $d$ looks
like an interesting project.  In particular, one may ask about 
other systematic sources of non-torsion points on $J$, as in
\cite{ConceicaoHallUlmer14}, and about the relative abundance or
scarcity of balanced rays for fixed $p$ and varying $r$ and $d$, as in
\cite{PomeranceUlmer13}.

\bibliography{database}{}
\bibliographystyle{alpha}

\vfil\eject

\section{Correction}\label{s:correction}
In this section (added after publication) we correct an error in
Section~5, and we discuss changes required in
the $p$-adic exercises of Section~7.  None of the main
results of the paper are affected.

\subsection{Counterexample}
Let $k$ be a finite field of characteristic $p$, let $W(k)$ be its
ring of Witt vectors, and let $\sigma:W(k)\to W(k)$ be its Frobenius
automorphism.  Let $\CC$ and $\DD$ be smooth, projective curves over
$k$.

Part (2) of Theorem~5.2 is incorrect as stated.  For a counterexample,
let $\CC=\DD=E$ where $E$ is an ordinary elliptic curve over $k=\Fq$,
and take $n=1$.  Using well-known properties of elliptic curves, one
finds that
$$\left(\left(H^1(\CC)\tensor_W H^1(\DD)\right)^{F=p}\right)/p
\cong\Fp^2$$
and
$$\left(H^1(\CC)/p^n\tensor_W H^1(\DD)/p\right)^{F=V=p}\cong\Fq^2.$$
On the other hand, it is known (Milne, Inv. Math. {\bf 6}, 1968,
p.~102) that the Brauer group of $E^2$ has order 
$$[\en_{\Fq}(E):\Z[\pi]]^2$$
where $\pi$ is the Frobenius endomorphism of $E$.  Since the
discriminant of $\Z[\pi]$ is prime to $p$, so is the displayed index,
and the $p$ part of the Brauer group is thus trivial.  This shows the
sequence in Theorem~5.2 part (2) is not exact in general.  The problem
turns out to be that the middle term is not correct.

\subsection{Corrected Theorem}
We reformulate Theorem~5.2.  Although part (1) is correct as stated,
we give an equivalent formulation which is more parallel to the
correct statement of part (2).

Let $W=W(k)$ and recall that $A$ denotes the Dieudonn\'e ring
$W\{F,V\}$ where $FV=VF=p$, $F\alpha=\sigma(\alpha)F$, and
$\alpha V=V\sigma(\alpha)$ for $\alpha\in W$.  Recall also that
$\NS'(\CC\times_k\DD)$ denotes the orthogonal complement in the
N\'eron-Severi group $\NS(\CC\times_k\DD)$ of the classes of
$P\times\DD$ and $\CC\times Q$ where $P$ and $Q$ are $k$-rational
divisors of degree 1 on $\CC$ and $\DD$ respectively.

\begin{thm}[Corrected Theorem~5.2]\label{thm:corrected-cohom-of-product-fixed}
\mbox{}
\begin{enumerate}
\item There is a functorial isomorphism
$$\NS'(\CC\times_k\DD)\tensor\Zp\isoto
\Hom_A\left(H^1(\DD),H^1(\CC)\right).$$
\item There is a functorial exact sequence
$$0\to \Hom_A\left(H^1(\DD),H^1(\CC)\right)/p^n\to
\Hom_A\left(H^1(\DD)/p^n,H^1(\CC)/p^n\right)
\to\Br(\CC\times_k\DD)_{p^n}\to0.$$
\end{enumerate}
Here $\Hom_A$ denotes homomorphisms of $A$-modules,
and ``functorial'' means that the displayed maps are equivariant for
the action of $\aut(\CC)\times\aut(\DD)$.
\end{thm}

The published proof of Theorem~5.2 minus the second half of the last
sentence proves the statement above.  The second
half of the last sentence, written in an overzealous desire for
symmetry, purports to go from the ``$\Hom$'' formulation above to a
``$\tensor$'' formulation, and this introduces an error.  Specifically,
the last displayed equation of the proof is not correct.  Omitting
this last translation yields the correct statement and proof.

\subsection{More on  Frobenius}
Before explaining the changes needed to the $p$-adic exercises of
Section~7, we add one detail on the action of Frobenius on the
crystalline ohomology group discussed in Section~6.

We use the notations of that section.  In particular, $\CC$ is the
smooth projective model of the affine curve $z^d=x^2-1$,
$H^1_{crys}(\CC/\Zp)$ is its first crystalline cohomology group, and
$e_i$ ($0<i<d$, $i\neq d/2$) is the basis of $H^1_{crys}(\CC/\Zp)$
appearing in Proposition~6.4.  (Here and below, we read the indices
modulo $d$.)  We showed that the action of Frobenius on
$H^1_{crys}(\CC/\Zp)$ is given by $F(e_i)=c_ie_{pi}$, where $c_i\in\Zp$
satisfies
$$\ord(c_i)=\begin{cases}
0&\text{if $i>d/2$}\\
1&\text{if $i<d/2$.}\end{cases}$$

\begin{lemma}\label{eq:c-i-minus-i}
$$c_ic_{-i}=\begin{cases}
p&\text{if $i<d/2$ and $pi<d/2$}\\
p&\text{if $i>d/2$ and $pi>d/2$}\\
-p&\text{if $i<d/2$ and $pi>d/2$}\\
-p&\text{if $i>d/2$ and $pi<d/2$}\\
\end{cases}$$
\end{lemma}

\begin{proof}
Let $f\in H^2(\CC/\Zp)$ be the cup product $e_1\cup e_{-1}$.  The
content of part (1) of Proposition~6.4 is that
$$e_i\cup e_{-i}=
\begin{cases}
f&\text{if $i<d/2$}\\
-f&\text{if $i>d/2$.}
\end{cases}$$
If $i<d/2$, we have
$$pf=F(f)=F(e_i\cup e_{-i})=c_ic_{-i}e_{pi}\cup e_{-pi},$$
and
$$e_{pi}\cup e_{-pi}=\begin{cases}
f&\text{if $pi<d/2$}\\
-f&\text{if $pi>d/2$.}
\end{cases}$$ 
Comparing the last two displays yields the first and
third cases of the lemma.  For the second and fourth, we have
$i>d/2$,
$$-pf=F(-f)=F(e_i\cup e_{-i})=c_ic_{-i}e_{pi}\cup e_{-pi},$$
and
$$e_{pi}\cup e_{-pi}=
\begin{cases}
f&\text{if $pi<d/2$}\\
-f&\text{if $pi>d/2$.}
\end{cases}$$
Comparing the last two displays yields the remaining two cases
of the lemma.
\end{proof}

Define
$$d_i:=c_i/p^{\ord_p(c_i)}.$$
We record two useful facts about the $d_i$:  First,
the $d_i$ are $p$-adic units, and by the lemma we have
\begin{equation}\label{eq:di-minus-i}
d_id_{-i}=\pm1
\end{equation}
where the sign is $+1$ if $i$ and $pi$ lie in the same half of the
interval $[0,d]$ and $-1$ if they lie in opposite halves.  

Second, by Proposition 6.4(4-5), if $o$ is an orbit of Frobenius on
$\Z/d\Z$ with $\gcd(d,o)<d/2$ and $p$ is balanced modulo $d/\gcd(d,o)$,
then
\begin{equation}\label{eq:prod-di}
\prod_{i\in o}d_i^2=1.
\end{equation}

\subsection{Modified $p$-adic exercises}
We now explain the changes needed in Sections~7.1, 7.2, and 7.4 when
we replace Theorem~5.2 with
Theorem~\ref{thm:corrected-cohom-of-product-fixed}.

We use the same notations as in Section~7.1:  
Write $W$ for the Witt vectors $W(\Fq)$, $W_n$ for $W/p^n$,
$H^1(\CC)$ for $H^1_{crys}(\CC/W)$, and $H^1(\DD)$ for
$H^1_{crys}(\DD/W)$ where $\CC=\DD$ is the curve over $\Fq$ studied in
Section~6.  The product $\CC\times_\Fq\DD$ carries an
action of $\Delta=\mu_2\times\mu_d$ acting ``anti-diagonally'' as well
as an action of $G=\mu_d\sdp\gal(\Fq/\Fp)$ acting on the factor $\CC$.

Our goal is to compute 
$$H':=\Hom_A\left(H^1(\CC),H^1(\DD)\right)^\Delta$$
and 
$$H'_n:=\Hom_A\left(H^1(\CC)/p^n,H^1(\DD)/p^n\right)^\Delta.$$  

For an orbit $o\in O_{d,p}$, we write $H^{\prime o}$ and
$H^{\prime o}_n$ for the $o$ parts of the corresponding groups, i.e.,
for the images of the projector $\pi_o$ of Section~2.8 on $H'$ or
$H'_n$.

Since
$$\Hom_A\left(H^1(\CC),H^1(\DD)\right)^\Delta\subset
\Hom_W\left(H^1(\CC),H^1(\DD)\right)^\Delta$$
and
\begin{align*}
\Hom_A\left(H^1(\CC)/p^n,H^1(\DD)/p^n\right)^\Delta&\subset
\Hom_W\left(H^1(\CC)/p^n,H^1(\DD)/p^n\right)^\Delta\\
&=\left(\Hom_W\left(H^1(\CC),H^1(\DD)\right)^\Delta\right)/p^n
\end{align*}
we first consider 
$$M':=\Hom_W(H^1(\CC),H^1(\DD))^\Delta.$$

Recall the $W$-basis $e_i$ ($0<i<d$, $i\neq d/2$) of
$H^1(\CC)=H^1(\DD)$ from Proposition~6.4.  For $0<i,j<d$,
$i,j\neq d/2$, let $\varphi_{ij}\in\Hom_W(H^1(\CC),H^1(\DD))$ be the
element with
$$\varphi_{ij}(e_\ell)=
\begin{cases}
e_i&\text{if $\ell=j$}\\
0&\text{if $\ell\neq j$.}
\end{cases}
$$
Then the $\varphi_{ij}$ form a $W$-basis of
$\Hom_W(H^1(\CC),H^1(\DD))$.  The submodule commuting with the
anti-diagonal action of $\Delta$, i.e.,
$M'$, is spanned by the $\varphi_{i,-i}$ with $0<i<d$, $i\neq d/2$.

Now we fix an orbit $o\in O_{d,p}$ and assume that $\gcd(o,d)<d/2$ and
that $p$ is balanced modulo $d/\gcd(o,d)$.  Let $i\in o$ be the
standard base point, and for $j=0,\dots,|o|-1$ define
$$f_{ip^j}=\left(\prod_{\ell=0}^{j-1}d_{ip^\ell}^{2}\right)\varphi_{ip^j,-ip^j}.$$
The $f_{ip^j}$ form a new basis of $M^{\prime o}$, the part of $M'$
cut out by the projector $\pi_o$.

It follows from equation~\eqref{eq:prod-di} that $f_{ip^j}$ only
depends on the class of $ip^j$ modulo $d$, i.e., we may read the
index $j$ modulo $|o|$ without any ambiguity.  

We now turn to computing $H^{\prime o}$ and $H_n^{\prime o}$.
Consider a typical element
$$c=\sum_{j=0}^{|o|-1}\alpha_jf_{ip^j}$$
where $\alpha_j\in W$ or $W_n$ and we read the index $j$ modulo $|o|$.

Applying $F\circ c$ and $c\circ F$ to $e_{-ip^j}$ for
$j=0,\dots,|o|-1$, we see that $F\circ c=c\circ F$ if and only if
\begin{equation}\label{eq:Fc=cF}
\sigma(\alpha_j)c_{ip^j}=\alpha_{j+1}d_{ip^j}^{2}c_{-ip^j}
\end{equation}
for all $j$.  A similar calcuation shows that $V\circ c=c\circ V$ if
and only if
$$\sigma(\alpha_j)\left(\frac p{c_{-ip^j}}\right)
=\alpha_{j+1}d_{ip^j}^{2}\left(\frac p{c_{ip^j}}\right),$$
and Proposition 6.4(4) and Lemma~\ref{eq:c-i-minus-i} show that this
equation is equivalent to \eqref{eq:Fc=cF}.

We now simplify equation~\eqref{eq:Fc=cF}, separating into four cases
depending on the positions of $ip^j$ and $ip^{j+1}$ in $[0,d]$.  More
precisely, recall the word $w=w_1\cdots w_{|o|}$ attached to $o$: the
letter $w_j$ is $l$ if the least positive residue of $ip^{j-1}$ modulo
$d$ is $>d/2$ and it is $u$ if the residue is $<d/2$.  Using
Proposition 6.4(4) and equation~\eqref{eq:di-minus-i}, we see that
equation~\eqref{eq:Fc=cF}  is equivalent to the equations
\begin{align*}
+\sigma(\alpha_j)&=p\alpha_{j+1}&\text{if }w_{j+1}w_{j+2}&=ll\\
-\sigma(\alpha_j)&=p\alpha_{j+1}&\text{if }w_{j+1}w_{j+2}&=lu\\
-p\sigma(\alpha_j)&=\alpha_{j+1}&\text{if }w_{j+1}w_{j+2}&=ul\\
+p\sigma(\alpha_j)&=\alpha_{j+1}&\text{if }w_{j+1}w_{j+2}&=uu.
\end{align*}

Note that when $w_{j+1}=l$, $\alpha_{j+1}$ determines $\alpha_j$, and when
$w_{j+1}=u$, $\alpha_j$ determines $\alpha_{j+1}$.  Thus we may eliminate
many of the variables $\alpha_j$.
More precisely, write the word $w$ in exponential form:
$w=u^{e_1}l^{e_2}\cdots l^{e_{2k}}$.  Setting $\beta_0=\alpha_0$ and
$$\beta_j=\alpha_{e_1+e_2+\cdots+e_{2j}}$$
for $1\le j\le k$ (so that $\beta_k=\beta_0$), the class $c$ is entirely
determined by the $\beta$'s.  Indeed, for $\sum_{i=1}^{2j}e_i\le \ell<
\sum_{i=1}^{2j+1}e_i$, we have
$$\alpha_\ell=(\sigma p)^{\ell-\sum_{i=1}^{2j}e_i}\beta_j$$
and for $\ell=\sum_{i=1}^{2j+1}e_i$, we have
$$\alpha_\ell=-(\sigma p)^{e_{2j+1}}\beta_j.$$
On the other hand, for
$\sum_{i=1}^{2j+1}e_i\le\ell<\sum_{i=1}^{2j+2}e_i$, we have
$$\alpha_\ell=-(\sigma^{-1}p)^{\sum_{i=1}^{2j+2}e_i-\ell}\beta_{j+1}$$
and for $\ell=\sum_{i=1}^{2j+2}e_i$, we have
$$\alpha_\ell=\beta_{j+1}.$$

The conditions on the $\alpha$'s translated to the $\beta$'s become
\begin{align}\label{eq:basic}
(\sigma p)^{e_1}\beta_0&=(\sigma^{-1}p)^{e_2}\beta_1\notag\\
(\sigma p)^{e_3}\beta_1&=(\sigma^{-1}p)^{e_4}\beta_2\notag\\
&\vdots\\
(\sigma p)^{e_{2k-1}}\beta_{k-1}&=(\sigma^{-1}p)^{e_{2k}}\beta_k\notag
\end{align}

These are exactly the ``basic equations'' (7.4.2) and the rest of the
calculation of $H^{\prime o}$ and $H_n^{\prime o}$ proceeds exactly as
in Sections~7.5 and 7.6.

\subsection{A few typos}
We take this opportunity to correct a few other typos.

In Proposition~6.4, part (2), ``$\lfloor(d+1)/2\rfloor$'' should be
``$\lceil(d+1)/2\rceil$''.

In the penultimate display of Section~7.2, on the right hand side,
``$f_{pi}$'' should be ``$f_{pj}$''.

In Section~7.4, all occurences of ``$w_j$'' should be ``$w_{j+1}$''.

In Proposition~7.5.1, ``$p$ is balanced modulo $p$'' should be ``$p$
is balanced modulo $d/\gcd(o,d)$''.

The sixth displayed equation in Section 7.6 is missing several
powers of $\sigma$.  It should read
\begin{align*}
0=p^{ht(o)}\beta_k&=p^{e_{2\ell+2,2k}}\beta_k\\
&=\sigma^{e_{2k-1}+e_{2k}}p^{e_{2\ell+2,2k-2}}\beta_{k-1}\\
&\ \ \vdots\\
&=\sigma^{e_{2\ell+3}+\cdots+e_{2k}}p^{e_{2\ell+2}}\beta_{\ell+1}.
\end{align*}
The key point, namely that $p^{e_{2\ell+2}}\beta_{\ell+1}=0$, is
unchanged.

\end{document}

%% file: formatting.tex
\numberwithin{equation}{subsection}

\swapnumbers
\theoremstyle{plain}
\newtheorem{thm}[subsection]{Theorem}
\newtheorem{thmss}[subsubsection]{Theorem}
\newtheorem{prop}[subsubsection]{Proposition}
\newtheorem{props}[subsection]{Proposition}
\newtheorem{lemma}[subsubsection]{Lemma}

\newtheorem{corss}[subsubsection]{Corollary}

\theoremstyle{definition}
\newtheorem{defn}[subsubsection]{Definition}

\theoremstyle{remark}
\newtheorem{rem}[subsubsection]{Remark}



%% file: macros.tex
\usepackage[OT2,T1]{fontenc}
\DeclareSymbolFont{cyrletters}{OT2}{wncyr}{m}{n}
\DeclareMathSymbol{\sha}{\mathalpha}{cyrletters}{"58}

\newcommand{\CC}{\mathcal{C}}

\newcommand{\DD}{\mathcal{D}}

\renewcommand{\SS}{\mathcal{S}}

\newcommand{\XX}{\mathcal{X}}

\newcommand{\WW}{{\mathcal{W}}}

\newcommand{\EE}{\mathcal{E}}

\newcommand{\OO}{\mathcal{O}}

\newcommand{\F}{\mathbb{F}}
\newcommand{\Fp}{{\mathbb{F}_p}}
\newcommand{\Fptimes}{{\mathbb{F}_p^\times}}
\newcommand{\Fq}{{\mathbb{F}_q}}

\newcommand{\Fpbar}{{\overline{\mathbb{F}}_p}}
\newcommand{\Fpbartimes}{{\overline{\mathbb{F}}_p^\times}}

\newcommand{\ratto}{{\dashrightarrow}}

\newcommand{\Qp}{{\mathbb{Q}_p}}
\newcommand{\Qpbar}{{\overline{\mathbb{Q}}_p}}
\newcommand{\Zp}{{\mathbb{Z}_p}}

\newcommand{\Z}{\mathbb{Z}}

\renewcommand{\P}{\mathbb{P}}

\newcommand{\f}{\mathfrak{f}}

\newcommand{\p}{\mathfrak{p}}

\newcommand{\<}{\langle}
\renewcommand{\>}{\rangle}
\newcommand{\into}{\hookrightarrow}

\newcommand{\isoto}{\,\tilde{\to}\,}
\newcommand{\longisoto}{\,\tilde{\longrightarrow}\,}
\newcommand{\tensor}{\otimes}

\newcommand{\sdp}{{\rtimes}}

\newcommand{\G}{\mathbb{G}}

\newcommand{\labeledto}[1]{\overset{#1}{\to}}

\newcommand{\pfrac}[2]{\genfrac{(}{)}{}{}{#1}{#2}}

\DeclareMathOperator{\inv}{inv}

\DeclareMathOperator{\ord}{ord}

\DeclareMathOperator{\Hom}{Hom}
\DeclareMathOperator{\aut}{Aut}
\DeclareMathOperator{\Pic}{Pic}
\DeclareMathOperator{\divcorr}{DivCorr}

\DeclareMathOperator{\NS}{NS}

\DeclareMathOperator{\Br}{Br}

\DeclareMathOperator{\lcm}{lcm}
\DeclareMathOperator{\gal}{Gal}
\DeclareMathOperator{\spec}{Spec}

\DeclareMathOperator{\en}{End}

\DeclareMathOperator{\disc}{Disc}

\DeclareMathOperator{\Sel}{Sel}
\DeclareMathOperator{\Fr}{Fr}

\DeclareMathOperator{\len}{Len}

\def\clap#1{\hbox to 0pt{\hss#1\hss}}

%% file: L3.bbl
\newcommand{\etalchar}[1]{$^{#1}$}
\def\cprime{$'$}
\begin{thebibliography}{BHP{\etalchar{+}}14}

\bibitem[Art74]{Artin74}
M.~Artin.
\newblock Supersingular {$K3$} surfaces.
\newblock {\em Ann. Sci. {\'e}cole Norm. Sup. (4)}, 7:543--567 (1975), 1974.

\bibitem[BHP{\etalchar{+}}14]{AIMgroup}
L.~Berger, C.~Hall, R.~Pannekoek, J.~Park, R.~Pries, S.~Sharif, A.~Silverberg,
  and D.~Ulmer.
\newblock Explicit high ranks in higher genus.
\newblock In preparation, 2014.

\bibitem[Bro97]{Broumas97}
A.~Broumas.
\newblock Effective {$p$}-descent.
\newblock {\em Compositio Math.}, 107:125--141, 1997.

\bibitem[CHU14]{ConceicaoHallUlmer14}
R.~P. Concei{\c{c}}{\~a}o, C.~Hall, and D.~Ulmer.
\newblock Explicit points on the {L}egendre curve {II}.
\newblock {\em Math. Res. Lett.}, 21:261--280, 2014.

\bibitem[Dum95]{Dummigan95}
N.~Dummigan.
\newblock The determinants of certain {M}ordell-{W}eil lattices.
\newblock {\em Amer. J. Math.}, 117:1409--1429, 1995.

\bibitem[Dum99]{Dummigan99}
N.~Dummigan.
\newblock Complete {$p$}-descent for {J}acobians of {H}ermitian curves.
\newblock {\em Compositio Math.}, 119:111--132, 1999.

\bibitem[Gra77]{Gras77}
G.~Gras.
\newblock Classes d'id{\'e}aux des corps ab{\'e}liens et nombres de {B}ernoulli
  g{\'e}n{\'e}ralis{\'e}s.
\newblock {\em Ann. Inst. Fourier (Grenoble)}, 27:1--66, 1977.

\bibitem[Gro61]{EGA3a}
A.~Grothendieck.
\newblock \'{E}l\'ements de g\'eom\'etrie alg\'ebrique. {III}. \'{E}tude
  cohomologique des faisceaux coh\'erents. {I}.
\newblock {\em Inst. Hautes \'Etudes Sci. Publ. Math.}, 11, 1961.

\bibitem[Gro68a]{Grothendieck68ii}
A.~Grothendieck.
\newblock Le groupe de {B}rauer. {II}. {T}h{\'e}orie cohomologique.
\newblock In {\em Dix {E}xpos{\'e}s sur la {C}ohomologie des {S}ch{\'e}mas},
  pages 67--87. North-Holland, Amsterdam; Masson, Paris, 1968.

\bibitem[Gro68b]{Grothendieck68iii}
A.~Grothendieck.
\newblock Le groupe de {B}rauer. {III}. {E}xemples et compl{\'e}ments.
\newblock In {\em Dix {E}xpos{\'e}s sur la {C}ohomologie des {S}ch{\'e}mas},
  pages 88--188. North-Holland, Amsterdam; Masson, Paris, 1968.

\bibitem[Hus04]{HusemollerEC}
D.~Husem{{\"o}}ller.
\newblock {\em Elliptic curves}, volume 111 of {\em Graduate Texts in
  Mathematics}.
\newblock Springer-Verlag, New York, second edition, 2004.

\bibitem[Ill79]{Illusie79b}
L.~Illusie.
\newblock Complexe de de\thinspace {R}ham-{W}itt et cohomologie cristalline.
\newblock {\em Ann. Sci. {\'e}cole Norm. Sup. (4)}, 12:501--661, 1979.

\bibitem[IR90]{IrelandRosenCIMNT}
K.~Ireland and M.~Rosen.
\newblock {\em A classical introduction to modern number theory}, volume~84 of
  {\em Graduate Texts in Mathematics}.
\newblock Springer-Verlag, New York, second edition, 1990.

\bibitem[Maz72]{Mazur72b}
B.~Mazur.
\newblock Frobenius and the {H}odge filtration.
\newblock {\em Bull. Amer. Math. Soc.}, 78:653--667, 1972.

\bibitem[Maz73]{Mazur73}
B.~Mazur.
\newblock Frobenius and the {H}odge filtration (estimates).
\newblock {\em Ann. of Math. (2)}, 98:58--95, 1973.

\bibitem[Mil75]{Milne75}
J.~S. Milne.
\newblock On a conjecture of {A}rtin and {T}ate.
\newblock {\em Ann. of Math. (2)}, 102:517--533, 1975.

\bibitem[Mil80]{MilneEC}
J.~S. Milne.
\newblock {\em Etale cohomology}, volume~33 of {\em Princeton Mathematical
  Series}.
\newblock Princeton University Press, Princeton, N.J., 1980.

\bibitem[MM74]{MazurMessingUEa1DCC}
B.~Mazur and W.~Messing.
\newblock {\em Universal extensions and one dimensional crystalline
  cohomology}.
\newblock Lecture Notes in Mathematics, Vol. 370. Springer-Verlag, Berlin-New
  York, 1974.

\bibitem[MW84]{MazurWiles84}
B.~Mazur and A.~Wiles.
\newblock Class fields of abelian extensions of {${\bf Q}$}.
\newblock {\em Invent. Math.}, 76:179--330, 1984.

\bibitem[PU13]{PomeranceUlmer13}
C.~Pomerance and D.~Ulmer.
\newblock On balanced subgroups of the multiplicative group.
\newblock In {\em Number Theory and Related Fields: In Memory of {A}lf van der
  {P}oorten}, pages 253--270. Springer, New York, 2013.

\bibitem[Shi91]{Shioda91}
T.~Shioda.
\newblock Mordell-{W}eil lattices and sphere packings.
\newblock {\em Amer. J. Math.}, 113:931--948, 1991.

\bibitem[Ulm91]{Ulmer91}
D.~L. Ulmer.
\newblock {$p$}-descent in characteristic {$p$}.
\newblock {\em Duke Math. J.}, 62:237--265, 1991.

\bibitem[Ulm11]{Ulmer11}
D.~Ulmer.
\newblock Elliptic curves over function fields.
\newblock In {\em Arithmetic of ${L}$-functions ({P}ark {C}ity, {UT}, 2009)},
  volume~18 of {\em IAS/Park City Math. Ser.}, pages 211--280. Amer. Math.
  Soc., Providence, RI, 2011.

\bibitem[Ulm13]{Ulmer13a}
D.~Ulmer.
\newblock On {M}ordell-{W}eil groups of {J}acobians over function fields.
\newblock {\em J. Inst. Math. Jussieu}, 12:1--29, 2013.

\bibitem[Ulm14a]{Ulmer14a}
D.~Ulmer.
\newblock Explicit points on the {L}egendre curve.
\newblock {\em J. Number Theory}, 136:165--194, 2014.

\bibitem[Ulm14b]{Ulmer14b}
D.~Ulmer.
\newblock Curves and {J}acobians over function fields.
\newblock In G.~Boeckle et~al., editor, {\em Arithmetic Geometry over Global
  Function Fields}, Advanced Courses in Mathematics CRM Barcelona, pages
  281--337. Springer, Basel, 2014.

\bibitem[WM71]{WaterhouseMilne71}
W.~C. Waterhouse and J.~S. Milne.
\newblock Abelian varieties over finite fields.
\newblock In {\em 1969 {N}umber {T}heory {I}nstitute ({P}roc. {S}ympos. {P}ure
  {M}ath., {V}ol. {XX}, {S}tate {U}niv. {N}ew {Y}ork, {S}tony {B}rook,
  {N}.{Y}., 1969)}, pages 53--64. Amer. Math. Soc., Providence, R.I., 1971.

\end{thebibliography}
